\documentclass[twoside,11pt]{article}

\usepackage{blindtext}

\usepackage[preprint]{jmlr2e}
\usepackage{natbib}

\usepackage{hyperref}       
\usepackage{url}            
\usepackage{booktabs}       
\usepackage{amsfonts}       
\usepackage{bbding}
\usepackage{xcolor}         
\usepackage{algorithm}
\usepackage{algorithmic}
\usepackage{graphicx}
\usepackage{enumitem}
\usepackage{microtype}
\usepackage{color}
\usepackage{amsmath}
\usepackage{amsfonts,bm}

\def\1{\bf{1}}
\let\norm\relax
\let\Norm\relax
\newcommand{\Norm}[1]{\left\| #1 \right\|}
\newcommand{\norm}[1]{\left\| #1 \right\|}

\def\vzero{{\bf{0}}}

\def\va{{\bf{a}}}

\def\ve{{\bf{e}}}

\def\vu{{\bf{u}}}
\def\vv{{\bf{v}}}
\def\vw{{\bf{w}}}
\def\vx{{\bf{x}}}
\def\vy{{\bf{y}}}
\def\vz{{\bf{z}}}

\def\fN{{\mathcal{N}}}

\def\fP{{\mathcal{P}}}

\def\fV{{\mathcal{V}}}

\def\BE{{\mathbb{E}}}

\def\BR{{\mathbb{R}}}

\def \P{{\bf{P}}}
\def\mA {{\bf A}}

\def\mE {{\bf E}}

\def\mG {{\bf G}}
\def\mH {{\bf H}}
\def\mI {{\bf I}}

\def\mL {{\bf L}}

\def\mP {{\bf P}}
\def\mQ {{\bf Q}}
\def\mR {{\bf R}}

\def\mU {{\bf U}}

\usepackage{etoolbox}


\newtheorem{theorem}{Theorem}[section]
\newtheorem{proposition}[theorem]{Proposition}
\newtheorem{lemma}[theorem]{Lemma}
\newtheorem{corollary}[theorem]{Corollary}
\newtheorem{definition}[theorem]{Definition}
\newtheorem{assumption}[theorem]{Assumption}
\newtheorem{remark}[theorem]{Remark}

\def\vx {{{\bf x}}}
\def\vw {{{\bf w}}}
\def\vy {{{\bf y}}}
\def\va {{{\bf a}}}
\def\BR{{\mathbb{R}}}

\def\A{{\bf A}}
\def\B{{\bf B}}

\def\E{{\bf E}}
\def\e{{\bf e}}

\def\G{{\bf G}}
\def\H{{\bf H}}
\def\I{{\bf I}}

\def\LL{{\bf L}}

\def\Q{{\bf Q}}

\def\R{{\bf R}}

\def\S{{\bf S}}

\def\U{{\bf U}}
\def\u{{\bf u}}
\def\V{{\bf V}}
\def\v{{\bf v}}

\def\w{{\bf w}}

\def\x{{\bf x}}

\def\y{{\bf y}}
\def\Z{{\bf Z}}
\def\z{{\bf z}}
\def\0{{\bf 0}}
\def\1{{\bf 1}}

\def\OM{{\mathcal O}}

\def\RB{{\mathbb R}}
\def\EB{{\mathbb E}}

\def \mSigma{{\bf \Sigma}}

\def \srk{{\rm \text{SR-$k$}}}

\def \bfgs{{\rm \text{BlockBFGS}}}
\def \dfp{{\rm \text{BlockDFP}}}
\usepackage{algorithm}
\usepackage{algorithmic}
\usepackage{enumitem}
\usepackage{hhline}

\def \EBP #1{\EB\big[#1\big]}
\def \EBcommon #1{\EB[#1]}

\def \exp #1{{\rm exp}\left(#1\right)}
\def \tr #1{{\rm tr}\big(#1\big)}
\def \trcommon #1{{\rm tr}(#1)}

\usepackage{colortbl}
\usepackage[para]{threeparttable}
\usepackage{tcolorbox}
\usepackage{pifont}
\definecolor{mydarkgreen}{RGB}{39,130,67}
\definecolor{mydarkred}{RGB}{192,25,25}
\definecolor{bgcolor}{rgb}{0.93,0.99,1}
\definecolor{bgcolor2}{rgb}{0.8,1,0.8}
\definecolor{bgcolor3}{rgb}{0.50,0.90,0.50}

\usepackage{lastpage}
\jmlrheading{27}{2026}{1-\pageref{LastPage}}{1/26; Revised
6/26}{7/26}{26-0038}{Chengchang Liu, Cheng Chen, Luo Luo}
\ShortHeadings{title}{Liu, Chen, Luo}

\ShortHeadings{Symmetric Rank-$k$ Methods}{Liu, Chen, and Luo}
\firstpageno{1}

\begin{document}

\title{Symmetric Rank-$k$ Methods}

\author{\name Chengchang Liu \email liuchengchang@westlake.edu.cn \\
       \addr Department of Artificial Intelligence, Westlake University\\
       Hangzhou, China       \AND
       \name Cheng Chen \email chchen@sei.ecnu.edu.cn \\
       \addr Software Engineering Institute, East China Normal University\\
       Shanghai, China
  \AND
       \name Luo Luo\thanks{The corresponding author} \email luoluo@fudan.edu.cn \\
       \addr School of Data Science, Fudan University\\
       Shanghai, China \\
       and \\
       Shanghai Key Laboratory for Contemporary Applied Mathematics \\
       Shanghai, China}       
\editor{Michael Mahoney}

\maketitle

\begin{abstract}
This paper proposes a novel class of block quasi-Newton methods for convex optimization which we call symmetric rank-$k$ (SR-$k$) methods.
Each iteration of SR-$k$ incorporates the curvature information with~$k$ Hessian-vector products achieved from the greedy or random strategy.
We prove that SR-$k$ methods have the local superlinear convergence rate of~$\OM\big((1-k/d)^{t(t-1)/2}\big)$ for minimizing smooth and strongly convex functions, where $d$ is the problem dimension and $t$ is the iteration counter.
This is the first explicit superlinear convergence rate for block quasi-Newton methods, 
and it successfully explains why block quasi-Newton methods converge faster than ordinary quasi-Newton methods in practice.
We also leverage the idea of SR-$k$ methods to study the block BFGS and block DFP methods, showing their superior convergence rates.
\end{abstract}


\section{Introduction}

We study quasi-Newton methods for solving the minimization problem
\begin{align}\label{prob:main}
    \min_{\x\in\RB^d} f(\x),
\end{align}
where $f:\BR^d\to\BR$ is smooth and strongly convex. 
Quasi-Newton methods~\citep{broyden1970convergence2,broyden1970convergence,shanno1970conditioning, broyden1967quasi,davidon1991variable,byrd1987global,yuan1991modified,asl2024aj} are widely recognized for their fast convergence rates and efficient updates, which attract growing attention in many fields such as statistics~\citep{jamshidian1997acceleration, zhang2011quasi,bishwal2007parameter},  economics~\citep{ludwig2007gauss,li2013dynamic} and machine learning~\citep{goldfarb2020practical, hennig2013quasi,liu2022quasi,liu2022partial,lee2018distributed,qiu2023quasi}.
Unlike standard Newton methods, which need to compute the Hessian and its inverse, quasi-Newton methods go along the descent direction by the following scheme
\begin{align*}
    \x_{t+1}=\x_t-\G_t^{-1}\nabla f(\x_t),
\end{align*}
where $\G_t\in\BR^{d\times d}$ is an estimator of the Hessian  $\nabla^2 f(\x_t)$.
The most popular ways to construct the Hessian estimator are the Broyden family updates, including the Broyden--Fletcher--Goldfarb--Shanno (BFGS) method~\citep{broyden1970convergence2,broyden1970convergence,shanno1970conditioning}, the Davidon--Fletcher--Powell (DFP) method~\citep{davidon1991variable,fletcher1963rapidly}, and the symmetric rank-one (SR1) method~\citep{broyden1967quasi,davidon1991variable}. 

The classical quasi-Newton methods with Broyden family updates  {\citep{broyden1970convergence,broyden1970convergence2}} find the Hessian estimator $\G_{t+1}$ for the next round by the secant equation 
\begin{align}\label{eq:sec}
\G_{t+1}(\x_{t+1}-\x_t)=\nabla f(\x_{t+1})-\nabla f(\x_t).    
\end{align}

These methods have been proven to exhibit local superlinear convergence in the 1970s \citep{powell1971on,Dennis1974A,broyden1973on,dai2002convergence},
and their non-asymptotic superlinear rates were established in recent years~\citep{rodomanov2021rates,rodomanov2021new,ye2022towards,jin2022non,jin2022sharpened}.
For example, Rodomanov and Nesterov~\citep{rodomanov2021rates} showed the classical BFGS method enjoys the local superlinear rates of $\OM\big((d\varkappa/t)^{t/2}\big)$, where $\varkappa>1$ is the condition number. They also improved this result to $\OM\big((\exp{d\ln(\varkappa)/t}-1)^{t/2}\big)$ in consequent work \citep{rodomanov2021new}. 
Later, \cite{ye2022towards} showed the classical SR1 method converges with a local superlinear rate of $\OM\big((d\ln(\varkappa)/t)^{t/2}\big)$.

Some recent work \citep{gower2017randomized,rodomanov2021greedy,lin2021greedy,ji2023greedy} proposed another type of quasi-Newton methods, which construct the Hessian estimator by the following equation 
\begin{align}
\label{eq:grracondi}
    \G_{t+1}=\nabla^2 f(\x_{t+1}){\vu}_{t}
\end{align}
where ${\bf u}_t\in\RB^{d}$ is chosen by a greedy or randomized strategy.
\cite{rodomanov2021greedy} established a local superlinear rate of $\OM\big((1-{1}/{(\varkappa d)})^{t(t-1)/2}\big)$ for greedy quasi-Newton methods with Broyden family updates.
Later, \cite{lin2021greedy} provided a condition-number free superlinear rate of~$\OM\big((1-{1}/{d})^{t(t-1)/2}\big)$ for the greedy and randomized quasi-Newton methods with the specific BFGS and SR1 updates.

Block quasi-Newton methods construct the Hessian estimator along multiple directions at each iteration. 
The study of these methods dates back to the 1980s. Schnabel \citep{schnabel1983quasi} proposed the first block BFGS method by extending equation (\ref{eq:sec}) to multiple secant equations, i.e.,
$\G_{t+1}(\x_{t+1}-\x_{t+1-j})=\nabla f(\x_{t+1})-\nabla f(\x_{t+1-j})$
for all~$j=1,\dots,k$.
However, the resulting update of this method cannot guarantee that the Hessian estimator is symmetric, and it requires additionally modification to ensure the positive definiteness~\citep{schnabel1983quasi,o1994linear,gao2018block,lee2023almost}. 
Another line of research in 2010s~\citep{gao2018block,gower2016stochastic,gower2017randomized,kovalev2020fast,boutet2020secant} considers the variants of block BFGS methods that generalize condition~\eqref{eq:grracondi} to
\begin{align}
\label{eq:blockcondi}
    \G_{t+1}\mU_t=\nabla^2 f(\x_{t+1})\mU_t,
\end{align}
where $\mU_t=[\vu_t^{(1)},\cdots,\vu_t^{(k)}]\in\RB^{d\times k}$ indicates multiple directions constructed by deterministic or randomized methods. 
Since the Hessian-matrix product $\nabla^2 f(\x_{t+1})\mU_t$ can be efficiently achieved  by accessing the Hessian-vector products $\{\nabla^2 f(\x_{t+1})\vu_t^{(j)}\}_{j=1}^k$ in parallel {~\citep{gao2018block}}, these block BFGS methods usually exhibit better empirical performance than the ordinary quasi-Newton methods without block fashion updates.
\cite{gao2018block, kovalev2020fast} provided asymptotic local superlinear rates for block BFGS methods based on the condition~\eqref{eq:blockcondi}, 
but the advantage of block quasi-Newton methods over ordinary quasi-Newton methods is still unclear in theory. 
This naturally leads to the following question:

\textit{Can we provide a block quasi-Newton method with explicit superior convergence rate over the ordinary ones?}

In this paper, we give an affirmative answer by proposing symmetric rank-$k$ \mbox{(SR-$k$)} methods.
We design a novel block update to construct the Hessian estimator \mbox{$\mG_{t+1}\in\BR^{d\times d}$} that satisfies equation~\eqref{eq:blockcondi}.
We also provide the randomized and greedy strategies to determine directions in $\mU_t\in\BR^{d\times k}$ for SR-$k$ methods, which lead to the explicit local superlinear convergence rate of~$\OM\big((1-{k}/{d})^{t(t-1)/2}\big)$, where $k$~is the number of directions used to approximate the Hessian at each iteration.
For the case of~$k=1$, our methods reduce to randomized and greedy SR1 methods~\citep{lin2021greedy}.
For the case of~$k\geq 2$, it is clear that SR-$k$ methods have faster superlinear rates than existing greedy and randomized quasi-Newton methods~\citep{lin2021greedy,rodomanov2021greedy}.
For the case of $k=d$, it recovers the quadratic convergence rate like standard Newton's method.
We also follow the design of SR-$k$ to propose the variants of randomized block BFGS methods~\citep{gower2016stochastic,gower2017randomized,kovalev2020fast} and a randomized block DFP method, resulting in faster explicit superlinear rates.
We compare the proposed methods with existing quasi-Newton methods in Table~\ref{tbl:main-min}. 


\paragraph{Paper Organization}
The remainder of this paper is organized as follows.
In Section~\ref{sec:pre}, we introduce notations and assumptions throughout this paper. 
In Section~\ref{sec:update}, we propose the SR-$k$ update in the view of matrix approximation. 
In Section~\ref{sec:algorithm}, we propose the quasi-Newton methods with SR-$k$ updates for minimizing smooth strongly convex functions and provide their superior local superlinear convergence rates. 
In Section~\ref{sec:Block BFGS}, we propose the new randomized block BFGS methods and randomized block DFP method with explicit local superlinear convergence rates.
In Section~\ref{sec:exp}, we conduct numerical experiments to show the outperformance of the proposed methods.
Finally, we conclude our work in Section~\ref{sec:conclusion}.

\begin{table*}[!t]
	\centering
	\caption{\small We summarize local convergence rates of existing and proposed quasi-Newton methods for strongly convex optimization, where $\varkappa$ denotes the condition number of the objective and $t$ is the iteration counter. 
    The second column displays the number of secant equations used to construct the Hessian estimator.
    \label{tbl:main-min}} \vskip-0.2cm	
\begin{threeparttable}
\setlength\tabcolsep{5.pt}
\begin{tabular}{c c c}
\toprule[.1em]
\begin{tabular}{c} 
    \bf Methods  
\end{tabular} &
\bf  \#Equations & 
\begin{tabular}{c}  \bf Convergence \end{tabular} \\
\toprule[.1em]

\begin{tabular}{c}  Classical BFGS  \\
{\scriptsize \citep{rodomanov2021rates,rodomanov2021new}}
\end{tabular}
& $1$ 
&\begin{tabular}{c}
$\displaystyle{\OM\Big(\Big(\frac{d\varkappa}{t}\Big)^{t/2}\Big)}$\\ [0.25cm]
$\displaystyle{\OM\Big(\Big({\rm exp}\Big(\frac{d\ln(\varkappa)}{t}\Big)-1\Big)^{t/2}\Big)}$  
\end{tabular}   
\\  
\midrule 

\begin{tabular}{c}  Classical BFGS/DFP   \\
{\scriptsize \citep{jin2022non}}
\end{tabular}
& $1$ 
&\begin{tabular}{c}
$\displaystyle{\OM\Big(\Big (\frac{1}{t}\Big)^{t/2}\Big)}$~\tnote{(*)}  
\end{tabular}   
\\  
\midrule 
\begin{tabular}{c}  Classical DFP  \\{\scriptsize
\citep{rodomanov2021rates}}
\end{tabular}
& $1$ 
&\begin{tabular}{c}
$\displaystyle{\OM\Big(\Big(\frac{d\varkappa^2}{t}\Big)^{t/2}\Big)}$
\end{tabular}   
\\  
\midrule 

\begin{tabular}{c} Classical SR1 \\
{\scriptsize \citep{ye2022towards}}
\end{tabular}
& $1$ 
&  $\displaystyle{\OM\Big(\Big(\frac{d\ln(\varkappa)}{t}\Big)^{t/2}\Big)}$  
\\  
\midrule           
\begin{tabular}{c}
    Greedy/Randomized Broyden\\
   {\scriptsize \citep{rodomanov2021greedy,lin2021greedy}}
\end{tabular} 
& $1$
& $\displaystyle{\OM\Big(\Big(1-\frac{1}{\varkappa d}\Big)^{{t(t-1)}/{2}}\Big)}$  
\\  
\midrule

\begin{tabular}{c}
    Randomized BFGS\\
    {\scriptsize\citep{lin2021greedy}}
\end{tabular}
& $1$ 
& $\displaystyle{\OM\Big(\Big(1-\frac{1}{d}\Big)^{{t(t-1)}/{2}}\Big)}$     
\\
\midrule 

\begin{tabular}{c}
    Greedy/Randomized SR1 \\
  {\scriptsize  \citep{lin2021greedy}}
\end{tabular}
& $1$
& $\displaystyle{\OM\Big(\Big(1-\frac{1}{d}\Big)^{{t(t-1)}/{2}}\Big)}$  
\\
\midrule

\begin{tabular}{c}
Block-BFGS \\
{\scriptsize \citep{gao2018block}}
\end{tabular} 
& $k\in[d]$ 
& asymptotic superlinear 
\\  
\midrule 
\begin{tabular}{c}
Randomized Block-BFGS (v1) \\
{\scriptsize \citep{kovalev2020fast,gower2017randomized}}
\end{tabular} 
& $k\in[d]$ 
& asymptotic superlinear
\\  
\midrule
\cellcolor{bgcolor}	\begin{tabular}{c}
Greedy/Randomized SR-$k$ 	\\
Algorithm~\ref{alg:SRK} 
\end{tabular}
\cellcolor{bgcolor}
& 
\begin{tabular}{c} \\[-0.25cm]
    ~~\,$k\in[d-1]$   \\[0.15cm]
    ~\,$k=d$  
\end{tabular}
\cellcolor{bgcolor}	
& \cellcolor{bgcolor}	
\begin{tabular}{c} 
$\displaystyle\OM\Big(\Big(1-\frac{k}{d}\Big)^{{t(t-1)}/{2}}\Big)$ \\[0.2cm]
quadratic
\end{tabular}

\\[0.25cm]
\midrule
\cellcolor{bgcolor}	
\begin{tabular}{c}
Randomized Block-BFGS/DFP\\
Algorithm~\ref{alg:bfgs}
\end{tabular}
&\cellcolor{bgcolor}	 \begin{tabular}{c} \\[-0.25cm]
    ~$k\in[d-1]$   \\[0.15cm]
    $k=d$  
\end{tabular}
& \cellcolor{bgcolor}	\begin{tabular}{c}
$\displaystyle\OM\Big(\Big(1-\frac{k}{d\varkappa}\Big)^{{t(t-1)}/{2}}\Big)$ \\[0.2cm]
quadratic
\end{tabular} 
\\
\midrule
\cellcolor{bgcolor}	
\begin{tabular}{c}
Faster Randomized Block-BFGS  
\\
Algorithm~\ref{alg:fasterbfgs}
\end{tabular}
\cellcolor{bgcolor}	
& 
\begin{tabular}{c} \\[-0.25cm]
    ~~~\,$k\in[d-1]$   \\[0.15cm]
    ~~\,$k=d$  
\end{tabular}
\cellcolor{bgcolor}	
& \cellcolor{bgcolor}	
\begin{tabular}{c}
$\displaystyle\OM\Big(\Big(1-\frac{k}{d}\Big)^{{t(t-1)}/{2}}\Big)$ \\[0.2cm]
quadratic
\end{tabular}

\\[0.25cm]
\bottomrule[.1em]
\end{tabular} 
\begin{tablenotes}
    	{\scriptsize     
  			\item [{(*)}] This convergence rate requires an additional assumption that the initial Hessian estimator be sufficiently close to the exact Hessian~\cite[Corollary 3]{jin2022non}. \\
     }
\end{tablenotes}
\end{threeparttable}\vskip-0.55cm	
\end{table*}

\section{Preliminaries}
\label{sec:pre}
We first introduce the notations used in this paper. We use~$\{\e_1,\cdots,\e_d\}$ to present the standard basis in space $\RB^d$ and let $\I_d\in\RB^{d\times d}$ be the identity matrix. 
We use $\S^{\dag}$ to denote the Moore-Penrose inverse of a matrix $\S$.
We denote the trace of a square matrix by ${\rm tr}(\cdot)$. 
We use $\|\cdot\|$ to present the spectral norm of a matrix and the Euclidean norm of a vector. Given a positive definite matrix $\A\in\BR^{d\times d}$, we denote the corresponding weighted norm as~$\|\vx\|_\mA\triangleq(\vx^{\top}\mA\vx)^{1/2}$ for some $\vx\in\BR^d$. 
We use the notation~$\Norm{\vx}_\vz$ to present~$\|\vx\|_{\nabla^2 f(\vz)}$ for positive definite Hessian $\nabla^2 f(\vz)$, if there is no ambiguity for the reference function~$f(\cdot)$.
We also define
$\mE_{k}(\mA)\triangleq[\ve_{i_1};\cdots;\ve_{i_k}]\in\BR^{d\times k}$,
where $i_1,\dots,i_k$ are the indices for the largest $k$ diagonal entries of matrix $\mA\in\BR^{d\times d}$.

Throughout this paper, we suppose the objective in problem (\ref{prob:main}) satisfies the following assumptions.
\begin{assumption}
\label{ass:smooth}
We assume the objective $f:\BR^d\to\BR$ is $L$-smooth, i.e., there exists some constant~$L\geq 0$ such that $\|\nabla f(\vx)-\nabla f(\vy)\|\leq L\|\vx-\vy\|$ for any $\vx,\vy\in\BR^d.$
\end{assumption}
\begin{assumption}
\label{ass:strongconvex}
We assume the objective $f:\BR^d\to\BR$ is $\mu$-strongly convex, i.e., there exists some constant~$\mu>0$ such that 
\begin{align*}
f(\lambda\vx+(1-\lambda)\vy)\leq \lambda f(\vx) + (1-\lambda)f(\vy)-\frac{\lambda(1-\lambda)\mu}{2}\|\vx-\vy\|^2.  
\end{align*}
for any $\vx,\vy\in\BR^d$ and $\lambda\in[0,1]$.
\end{assumption}
We define the condition number as $\varkappa \triangleq L/\mu$. 
The following proposition shows the twice differentiable function has a bounded Hessian under Assumptions \ref{ass:smooth} and~\ref{ass:strongconvex}~\citep{nesterov2018lectures}.
\begin{proposition}
Suppose the objective $f:\BR^d\to\BR$ is twice differentiable and satisfies Assumptions~\ref{ass:smooth} and \ref{ass:strongconvex}, then it holds $\mu\mI_d \preceq \nabla^2 f(\x) \preceq L\mI_d$
for any $\x\in\BR^d$.
\end{proposition}
We also impose the assumption of strongly self-concordance~\citep{rodomanov2021greedy,lin2021greedy} as follows.
\begin{assumption}
\label{ass:strongself}
We assume the objective $f:\RB^d\to\RB$ is $M$-strongly self-concordant, i.e., there exists some constant $M\geq 0$ such that 
\begin{align*}
   \nabla^2 f(\y)-\nabla^2 f(\x)\preceq M\|\y-\x\|_\z\nabla^2 f(\w) 
\end{align*}
holds for all $\x,\y,\w,\z\in\RB^d$.
\end{assumption}

Noticing that the strongly convex function with Lipschitz continuous Hessian is strongly self-concordant~\citep{rodomanov2021greedy}.

\begin{proposition}
Suppose the objective $f:\BR^d\to\BR$ satisfies Assumption \ref{ass:strongconvex} and has $L_2$-Lipschitz continuous Hessian, i.e., for some $L_2\geq 0$, $\|\nabla^2 f(\x)-\nabla^2 f(\y)\|\leq L_2\|\x-\y\|$ holds
for all $\x$, $\y\in\RB^d$, then the function~$f(\cdot)$ is $M$-strongly self-concordant with $M={L_2}/{\mu^{3/2}}$.
\end{proposition}

\section{Symmetric Rank-$k$ Updates}
\label{sec:update}
We propose the symmetric rank-$k$ (SR-$k$) update for matrix approximation as follows.
\begin{definition}[SR-$k$ Update]
Let $\mA\in\BR^{d\times d}$ and $\mG\in\BR^{d\times d}$ be two positive-definite matrices with~$\A\preceq \G$. 
For any matrix $\U\in\RB^{d\times k}$, we define
\begin{align}
\label{eq:srk}    
    {\text{\rm SR-$k$}}(\G,\A,\U) \triangleq \G-(\G-\A)\U\big(\U^{\top}(\G-\A)\U\big)^{\dag}\U^{\top}(\G-\A).
\end{align}
\end{definition}
Noticing that the formula of \srk~update contains a term of Moore-Penrose inverse, since the matrix $\U^{\top}(\G-\A)\U$ is possibly singular even if matrices $\mA,\mG\in\BR^{d\times d}$, and~$\mU\in\BR^{d\times k}$ are full rank.
This leads to its design and analysis be quite different from other block-type updates, i.e., block BFGS/DFP updates we study in Section~\ref{sec:Block BFGS}.

In this section, 
we always use $\G_{+}$ to denote the output of \srk~update such that 
\begin{align}\label{iter:srk}
\G_{+}\triangleq\srk(\G,\A,\U).     
\end{align}
We provide the following lemma to show that \srk~update does not increase the deviation from target matrix $\A$. 
\begin{lemma}
\label{lm:sr1good}
Given any positive-definite matrices $\mA\in\BR^{d\times d}$, $\mG\in\BR^{d\times d}$ with~$\A\preceq\G\preceq\eta \A$ for some~$\eta\geq 1$, then $\G_{+}$ defined by \eqref{iter:srk} holds that
\begin{align}\label{eq:srk_good}
   \A\preceq\G_{+}\preceq \eta\A. 
\end{align}
\end{lemma}
\begin{proof}
According to the update rule (\ref{eq:srk}), we have
\begin{align*}
    &\G_{+}-\A= (\G-\A)-(\G-\A)\U(\U^{\top}(\G-\A)\U)^{\dag}\U^{\top}(\G-\A)\\
    &=\left(\I_d\!-\!(\G\!-\!\A)\U(\U^{\top}(\G\!-\!\A)\U)^{\dag}\U^{\top}\right)(\G-\A)\left(\I_d\!-\!\U(\U^{\top}(\G\!-\!\A)\U)^{\dag}\U^{\top}(\G\!-\!\A)\right)\\
    &\succeq\0, 
\end{align*}
which indicates $\G_{+}\succeq \A$.
On the other hand, the condition $\mG\preceq\eta\mA$ indicates
\begin{align*}
    \G_{+}&\preceq \eta\A - (\G-\A)\U(\U^{\top}(\G-\A)\U)^{\dag}\U^{\top}(\G-\A)\preceq \eta\A,
\end{align*}
where the last inequality is by the condition $\G\succeq \A$.
Hence, we finish the proof.
\end{proof}
To evaluate the convergence of SR-$k$ update, we introduce the quantity~\citep{lin2021greedy,ye2022towards}
 \begin{align}
 \label{eq:measure_srk}
     \tau_\A(\G)\triangleq\trcommon{\G-\A},
 \end{align}
to describe the difference between the target matrix $\mA$ and the current estimator $\mG$.

In the remainder of this section, we aim to establish the following convergence guarantee  
\begin{align}
\label{eq:srk-aim}
    \EBcommon{\tau_{\A}(\G_{+})}\leq \left(1-\frac{k}{d}\right)\tau_{\A}(\G)
\end{align}
for approximating the target matrix $\mA\in\BR^{d\times d}$ by \srk~iteration (\ref{iter:srk}) with the appropriate choice of (random) matrix $\mU\in\BR^{d\times k}$.
Note that if we take $k=1$, the equation (\ref{eq:srk-aim}) will reduce to $ \EBcommon{\tau_{\A}(\G_{+})}\leq \left(1-{1}/{d}\right)\tau_{\A}(\G)$, which corresponds to the convergence result of randomized and greedy SR1 updates~\cite{lin2021greedy}.

Observe that we can split $\tau_{\A}(\G_{+})$ as follows
\begin{align}
\label{eq:tauAG+}
\begin{split}
    \tau_{\A}(\G_{+}) &\!\overset{\eqref{eq:srk},\, \eqref{iter:srk}}{=}\tr{\G-\A - (\G-\A)\U(\U^{\top}(\G-\A)\U)^{\dag}\U^{\top}(\G-\A)} \\
    &\,\,\,\,=\,\,\,\,\,\underbrace{\trcommon{\G-\A}}_{\text{Part}~{\rm \uppercase\expandafter{\romannumeral1}}} - \underbrace{\tr{(\G-\A)\U(\U^{\top}(\G-\A)\U)^{\dag}\U^{\top}(\G-\A)}}_{\text{Part}~{\rm \uppercase\expandafter{\romannumeral2}}}.
    \end{split}
\end{align}
Since Part I is equal to $\tau_{\A}(\G)$ and Part II is nonnegative, the \srk~update can reduce the estimation error with regard to $\tau_\A(\cdot)$.
To obtain \eqref{eq:srk-aim}, we only need to prove 
\begin{align*}
\EBP{\tr{(\G-\A)\U(\U^{\top}(\G-\A)\U)^{\dag}\U^{\top}(\G-\A)}}\geq \frac{k}{d}\trcommon{\G-\A}.
\end{align*}
We first provide the following lemma to bound Part II (in the view of $\mR=\mG-\mA$). 
\begin{lemma}
\label{lm:pdneq2}
For a symmetric positive semi-definite matrix $\R\in\RB^{d\times d}$ and a full rank matrix $\U\in\RB^{d\times k}$ with $k\leq d$, it holds that
\begin{align}
\label{eq:keyeq}  \tr{\R\U\left(\U^{\top}\R\U\right)^{\dag}\U^{\top}\R} \geq \tr{\U\left(\U^{\top}\U\right)^{-1}\U^{\top}\R}.
\end{align}
\end{lemma}
\begin{proof}
Let $\U=\Q\mSigma\V^{\top}$ be the reduced SVD of $\mU\in\BR^{d\times k}$, where $\Q\in\RB^{d\times k}$, $\V\in\RB^{k\times k}$ are (column) orthonormal (i.e. $\Q^{\top}\Q=\I_k$ and $\V^{\top}\V=\I_k$) and~$\mSigma\in\RB^{k\times k}$ is diagonal, then the right-hand side of inequality (\ref{eq:keyeq}) can be written as
\begin{align*}
\tr{\U\left(\U^{\top}\U\right)^{-1}\U^{\top}\R}& = \tr{\Q\mSigma\V^{\top}\left(\V\mSigma^2\V^{\top}\right)^{-1}\V\mSigma\Q^{\top}\R}  =\tr{\Q\Q^{\top}\R}.
\end{align*}
Consequently, we upper bound the left-hand side of inequality (\ref{eq:keyeq}) as follows 
\begin{align*}
    &\tr{\R\U\left(\U^{\top}\R\U\right)^{\dag}\U^{\top}\R} = \tr{\R\Q\mSigma\V^{\top}\left(\V\mSigma\Q^{\top}\R\Q\mSigma\V^{\top}\right)^{\dag}\V\mSigma\Q^{\top}\R}\\    &=\tr{\Q^{\top}\R\Q\mSigma\V^{\top}\left(\V\mSigma\Q^{\top}\R\Q\mSigma\V^{\top}\right)^{\dag}\V\mSigma\Q^{\top}\R\Q}\\
    &~~~+\tr{\big(\I_d-\Q\Q^{\top}\big)^{1/2}\R\Q\mSigma\V^{\top}\left(\V\mSigma\Q^{\top}\R\Q\mSigma\V^{\top}\right)^{\dag}\V\mSigma\Q^{\top}\R\big(\I_d-\Q\Q^{\top}\big)^{1/2}}
    \\
    &\geq \tr{\Q^{\top}\R\Q\mSigma\V^{\top}\left(\V\mSigma\Q^{\top}\R\Q\mSigma\V^{\top}\right)^{\dag}\V\mSigma\Q^{\top}\R\Q}\\
&=\tr{(\V\mSigma)^{-1}\V\mSigma\Q^{\top}\R\Q\mSigma\V^{\top}\left(\V\mSigma\Q^{\top}\R\Q\mSigma\V^{\top}\right)^{\dag}\V\mSigma\Q^{\top}\R\Q\mSigma\V^\top\big(\mSigma\V^\top\big)^{-1}}\\
&=  \tr{(\V\mSigma)^{-1}\V\mSigma\Q^{\top}\R\Q\mSigma\V^\top\big(\mSigma\V^\top\big)^{-1}} 
= \tr{\Q^{\top}\R\Q},
\end{align*}
where 
the inequality is due to  
$\Q$ is column orthonormal (thus $\I_d-\Q\Q^{\top}\succeq \0$) and  we have $\R\Q\mSigma\V^\top\left(\V\mSigma\Q^{\top}\R\Q\mSigma\V^\top\right)^{\dag}\V\mSigma\Q^{\top}\R$ is positive semi-definite.
We connect above results to finish the proof. 
\end{proof}
We  provide two strategies for selecting matrix $\mU\in\BR^{d\times k}$ of the \srk~update:
\begin{enumerate}[label=(\alph*),topsep=1.5pt, leftmargin=1.2cm,itemsep=0.12cm]
\item For the randomized strategy, we construct matrix $\mU\in\BR^{d\times k}$ by sampling each of its entries  
 according to the standard normal distribution independently, i.e., $ [\mU]_{ij} \overset{{\rm i.i.d}}{\sim} \fN(0,1)$ for all $i\in[d]$ and $j\in[k]$.
\item For the greedy strategy, we construct matrix  $\mU\in\BR^{d\times k}$ as $\U=\mE_{k}(\G-\A)$,
where $\E_k(\cdot)$ follows the definition in Section \ref{sec:pre}. 
\end{enumerate}
The following lemma indicates that the term of  $\tr{\U(\U^{\top}\U)^{-1}\U\R}$ in inequality \eqref{eq:keyeq} can be further lower bounded when we choose $\U\in\BR^{d\times k}$ by the above randomized or greedy strategy, which guarantees a sufficient decrease of~$\tau_{\A}(\cdot)$ for \srk~updates.
\begin{lemma}
\label{lm:explicitbound}
\srk~updates with randomized and greedy strategies  have the following properties:
\begin{enumerate}[label=(\alph*),topsep=0pt, leftmargin=1.2cm,itemsep=0.15cm]
\item If $\U\in\BR^{d\times k}$ is chosen as $[\mU]_{ij} \overset{{\rm i.i.d}}{\sim} \fN(0,1)$, then we have    
\begin{align}\label{eq:random-up}         \EBP{\tr{\U\left(\U^{\top}\U\right)^{-1}\U^{\top}\R}} = \frac{k}{d} \trcommon{\R}
\end{align}
for any matrix $\R\in\BR^{d\times d}$.     
\item If $\U\in\BR^{d\times k}$ is chosen as $\U=\E_k(\R)$, then we have
\begin{align}\label{eq:greedy-up}
{\tr{\U\left(\U^{\top}\U\right)^{-1}\U^{\top}\R}} \geq \frac{k}{d} \trcommon{\R}
\end{align}
for any matrix $\R\in\BR^{d\times d}$.
\end{enumerate}
\end{lemma}
\begin{proof}
We first consider the randomized strategy such that each entry of the matrix $\mU\in\BR^{d\times k}$ is independently sampled from~$\fN(0,1)$.
We use $\fV_{d,k}$ to present the Stiefel manifold which is the set of all $d\times k$ column orthogonal matrices and denote~$\fP_{k,d-k}$ as the set of all $m\times m$ orthogonal projection matrices of rank $k$.

According to Theorem 2.2.1 (iii) and Theorem 2.2.2 (iii) of \cite{chikuse2003statistics}, the random matrix $ \Z=\U(\U^{\top}\U)^{-1/2}$
is uniformly distributed on $\fV_{d,k}$ and the random matrix
$  \Z\Z^\top=\U(\U^{\top}\U)^{-1}\U^{\top} $
is uniformly distributed on $\fP_{k,d-k}$.
Then, applying Theorem~2.2.2~(i) of \cite{chikuse2003statistics} on matrix $\Z\Z^{\top}$ achieves
\begin{align}\label{eq:EP}    \EB\left[\U(\U^{\top}\U)^{-1}\U^{\top}\right] = \frac{k}{d}\I_d.
\end{align}
Consequently, we have
\begin{align*}
    \EB \left[\tr{\U\left(\U^{\top}\U\right)^{-1}\U^{\top}\R}\right]&=\tr{\EB\left[\U\left(\U^{\top}\U\right)^{-1}\U^{\top}\right]\R}\overset{\eqref{eq:EP}}{=}\frac{k}{d}\trcommon{\R}.
\end{align*}
Then we consider the greedy strategy such that $\U=\E_k(\R)$.
We use $\{r_{i}\}_{i=1}^{k}$ to denote $k$ largest diagonal entries of $\R$ with $ r_{1}\geq r_{2} \geq \cdots \geq r_{k}$, then we have
\begin{align}
\label{eq:bigeq}
    \sum_{i=1}^{k}r_{i}\geq\frac{k}{d}\trcommon{\R}.
\end{align}
We let $\vu_i\in\BR^d$ be the $i$-th column of $\mU\in\BR^{d\times k}$, then we have
\begin{align*}  
&\tr{\U\big(\U^{\top}\U\big)^{-1}\U^{\top}\R} 
=\tr{\big(\U^{\top}\U\big)^{-1}\U^{\top}\R\U}
=\tr{\I_k\U^{\top}\R\U} =\tr{\U^{\top}\R\U}\\
&= \sum_{i=1}^{k} \u_{i}^{\top}\R \u_{i} =\sum_{i=1}^{k}r_{i}\!\!\overset{\eqref{eq:bigeq}}{\geq} \frac{k}{d}\trcommon{\R}.
\end{align*}
\end{proof}
\vspace{-2em}
\begin{remark}
The proof of result (\ref{eq:random-up}) in Lemma~\ref{lm:explicitbound}  uses some statistical results on manifold \citep{chikuse2003statistics}. For readers who are not familiar with manifold theory,
we also provide an elementary proof for equation (\ref{eq:random-up}) in Appendix~\ref{appen:addiproof}.
\end{remark}
Now, we formally present the convergence result~\eqref{eq:srk-aim} for \srk~update.
\begin{theorem}\label{thm:matrix}
Let $ \G_{+}={\text{\rm SR-$k$}}(\G,\A,\U)$
with $\G,\mA\in\RB^{d\times d}$ such that $\mG\succeq\mA$ and select $\U\in\RB^{d\times k}$, where $k\leq d$, by the randomized strategy $[\mU]_{ij} \overset{{\rm i.i.d}}{\sim} \fN(0,1)$ or the greedy strategy~$\mU=\mE_k(\G-\A)$.
Then we have $\EB\left[\tau_{\A}(\G_{+})\right]\leq \left(1-k/d\right)\tau_\A(\G).  $
\end{theorem}
\begin{proof}
Applying Lemma~\ref{lm:pdneq2} and Lemma~\ref{lm:explicitbound} by taking $\R=\G-\A$, we have
\begin{align*}
    \EBcommon{\tau_{\A}(\G_{+})} & ~~~\overset{\eqref{eq:tauAG+}}{=}~~~\tau_{\A}(\G)-\EBP{\tr{(\G-\A)\U(\U^{\top}\big(\G-\A)\U\big)^{\dag}\U^{\top}(\G-\A)}}\\
    & ~~~\overset{\eqref{eq:keyeq}}{\leq}~~~\tau_{\A}(\G) -\EBP{\tr{\U\left(\U^{\top}\U\right)^{-1}\U^{\top}(\G-\A)}}\\
    &~ \overset{\eqref{eq:random-up},\,\eqref{eq:greedy-up}}{\leq} \tau_{\A}(\G) -\frac{k}{d}\tr{(\G-\A)} 
    = \left(1-\frac{k}{d}\right)\tau_{\A}(\G).
\end{align*}
\end{proof}
Theorem \ref{thm:matrix} reveals the advantage of the \srk~update, i.e., we can achieve faster convergence with respect to $\tau_{\A}(\cdot)$ by increasing the block size $k$.
As a comparison, the results of ordinary randomized or greedy SR1 updates~\citep{lin2021greedy} are exactly the special case of Theorem~\ref{thm:matrix} when $k=1$.

\section{Minimizing Strongly Convex Function}\label{sec:srk-opt}
By leveraging the proposed \srk{ } updates, we introduce a novel block quasi-Newton method, referred to as the \srk{} method.
The specifics of the \srk~method are outlined in Algorithm~\ref{alg:SRK},
where $M>0$ is the self-concordant parameter that follows the notation in Assumption~\ref{ass:strongself}.

We shall consider the convergence rate of \srk~method (Algorithm~\ref{alg:SRK}) and show its superiority to existing quasi-Newton methods. 
Our convergence analysis is based on the measure of local gradient norm \citep{nesterov2018lectures,rodomanov2021greedy,lin2021greedy}, which is defined as 
\begin{align*}
\lambda(\x)\triangleq \sqrt{\nabla f(\x)^{\top}(\nabla^2f(\x))^{-1}\nabla f(\x)}.
\end{align*}
The analysis starts from the following result for quasi-Newton iterations.

\begin{lemma}[\!\!{\cite[Lemma 4.3]{rodomanov2021greedy}}]
\label{lm:linear-quadra}
Suppose that the twice differentiable objective $f:\BR^d\to\BR$ is strongly self-concordant with constant $M\geq0$ and the positive definite matrix $\G_t\in\BR^{d\times d}$ satisfies $ \nabla^2 f(\x_t)\preceq \G_t\preceq \eta_t \nabla^2 f(\x_t)$
for some $\eta_t\geq 1$ and $M\lambda(\x_t)\leq 2$.
Then the update formula 
\begin{align} \label{eq:iterbyG}
\x_{t+1}=\x_t-\G_t^{-1}\nabla f(\x_t)    
\end{align}
holds that 
\begin{align}
\label{eq:linear-quadra}
r_{t}\leq \lambda(\x_t)~~~\text{and}~~~  \lambda(\x_{t+1})\leq \left(1-\frac{1}{\eta_t}\right)\lambda(\x_t) + \frac{M(\lambda(\x_t))^2}{2}+\frac{M^2(\lambda (\x_t))^3}{4\eta_t},
\end{align}
where $r_t  = \|\x_{t+1}-\x_t\|_{\x_t}$.
\end{lemma}

Note that the value of $\eta_t$ in equation (\ref{eq:linear-quadra}) is crucial to describe the local convergence rates of different types of quasi-Newton methods.
Applying Lemma \ref{lm:linear-quadra} by setting $\eta_t = 3\eta_0/2$ and combining with Lemma~\ref{lm:sr1good}, we can establish the linear convergence rate of \srk~methods as follows (see Appendix~\ref{sec:srklinear} for the detailed proof).
 
\begin{theorem}\label{thm:srklinear}
Under Assumptions \ref{ass:smooth}, \ref{ass:strongconvex}, and \ref{ass:strongself}, if we run SR-k method (Algorithm~\ref{alg:SRK}) with $\vx_0\in\BR^d$, $\mG_0\in\BR^{d\times d}$ such that $\nabla^2 f(\x_0)\preceq \G_0\preceq \eta_0 \nabla^2f(\x_0)$ and  $M\lambda(\vx_0)\!\leq\! {\ln(3/2)}/(4\eta_0)$
for some $\eta_0\geq 1$. Then it holds that
\begin{align}
\label{eq:srklinear}
  \nabla^2 f(\x_{t})\preceq \G_t\preceq \frac{3\eta_0}{2}\nabla^2 f(\x_t)\qquad\text{and}\qquad\lambda(\x_t)\leq \left(1-\frac{1}{2\eta_0}\right)^t\lambda(\x_0).
\end{align}
\end{theorem}

\label{sec:algorithm}
\begin{algorithm}[t]
\caption{Symmetric Rank-$k$ Method}\label{alg:SRK}
\begin{algorithmic}[1]
\STATE \textbf{Input:} $\vx_0, \G_0$, and $k$ \\
\STATE \textbf{for} $t=0,1\dots$ \\
\STATE \quad $\x_{t+1}=\x_t-\G_t^{-1}\nabla f(\x_t)$ \\ 
\STATE \quad $r_t=\|\x_{t+1}-\x_{t}\|_{\x_t}$ \\
\STATE \quad  $\tilde{\G}_{t}=(1+Mr_t)\G_t$ \\
\STATE \quad Construct $\U_t\in\RB^{d\times k}$ by  \\
\quad\quad  (a) randomized strategy: $\left[\U_{t}\right]_{ij}\overset{\rm{i.i.d}}{\sim}{\fN(0,1)}$ \\
\quad\quad  (b) greedy strategy: $\U_t=\E_k(\tilde{\G}_t-\nabla^2 f(\x_{t+1}))$ \label{line:greedy} \\
\STATE \quad $\G_{t+1}= \srk(\tilde{\G}_t,\nabla^2f(\x_{t+1}),\U_t)$ \\
\STATE \textbf{end for}
\end{algorithmic}
\end{algorithm}\vskip-0.3cm

Furthermore, we can obtain the superlinear rate for iteration \eqref{eq:iterbyG} if there exists some sequence $\{\eta_t\}$ such that $\eta_t\geq 1$ for all $t\geq 1$ and $\lim_{t\to\infty}\eta_t=1$.
For example, the randomized and greedy SR1 methods~\citep{lin2021greedy} lead to some $\{\eta_t\}$ such that~$\eta_t\geq 1$ and 
$\EB[\eta_t-1]\leq \OM((1-1/d)^{t})$.
As the results shown in Theorem \ref{thm:matrix}, the proposed \srk~updates have superiority in matrix approximation. 
So we desire to construct some~$\{\eta_t\}$ for \srk~methods with the tighter bound
$\EB[\eta_t-1]\leq \OM((1-{k}/{d})^{t})$.

Based on the above intuition, we derive the faster local superlinear convergence rate for \srk{} methods in the following theorem (see Appendix~\ref{sec:srk_proof} for the detailed proof), which is explicitly sharper than the convergence rates of existing randomized and greedy quasi-Newton methods~\citep{lin2021greedy,rodomanov2021greedy}.
\begin{theorem}
\label{thm:srk}
Under Assumptions \ref{ass:smooth}, \ref{ass:strongconvex}, and \ref{ass:strongself}, we run SR-k method (Algorithm~\ref{alg:SRK}) with $\vx_0\in\BR^d$ and $\mG_0\in\BR^{d\times d}$ such that 
\begin{align}
\label{eq:initial}
   M \lambda(\x_0)\leq \frac{\ln 2}{2} \cdot \frac{d-k}{ \eta_0 d^2\varkappa}\qquad\text{and}\qquad\nabla^2 f(\x_0)\preceq\G_0\preceq \eta_0\nabla^2f(\x_0)
\end{align} 
for some $k<d$ and $\eta_0\geq 1$. 
Then it holds that
\begin{align}
\label{eq:E_lambda_srk}
     \BE\left[\frac{\lambda(\x_{t+1})}{\lambda(\x_t)}\right]\leq 2d\varkappa\eta_0\left(1-\frac{k}{d}\right)^{t},
\end{align}
which naturally indicates the following two-stage convergence behaviors:
\begin{enumerate}[label=(\alph*),topsep=0.05cm, itemsep=0.1cm, leftmargin=0.9cm]
\item For \srk~method with randomized strategy, we have
    \begin{align*}
    \lambda(\x_{t_0+t})\leq\left(1-\frac{k}{d+k}\right)^{t(t-1)/2}\cdot\left(\frac{1}{2}\right)^t\cdot\left(1-\frac{1}{2\eta_0}\right)^{t_0}\lambda(\x_0),
\end{align*}
with probability at least $1-\delta$ for some $\delta\in(0,1)$, where $t_0=\OM(d\ln(\varkappa d\eta_0/\delta)/k)$. 
\item For \srk~method with greedy strategy, we have
    \begin{align*}
    \lambda(\x_{t_0+t})\leq\left(1-\frac{k}{d}\right)^{t(t-1)/2}\cdot\left(\frac{1}{2}\right)^t\cdot\left(1-\frac{1}{2\eta_0}\right)^{t_0}\lambda(\x_0),
\end{align*}
where $t_0=\OM\left(d\ln(\eta_0d\varkappa )/k\right)$.
\end{enumerate}
\end{theorem}

\begin{remark}
The condition $\nabla^2 f(\x_0)\preceq\G_0\preceq \eta_0\nabla^2f(\x_0)$ in Theorem~\ref{thm:srklinear} and \ref{thm:srk} can be satisfied by simply setting $\mG_0=L\mI_d$, then we can efficiently implement the update on $\mG_t$ by Woodbury identity~\citep{woodbury1950inverting}.
In a follow-up work, \cite{liu2023block} analyzed block quasi-Newton methods for solving nonlinear equations. However, their convergence guarantees require the Jacobian estimator at the initial point to be sufficiently accurate, which leads to potentially expensive costs in the step of initialization. 
\end{remark}
 {
\begin{remark}
Theorem~\ref{thm:srk} also indicates the theoretical difference between the randomized and greedy strategies.
The greedy strategy enjoys a deterministic convergence guarantee with respect to the local gradient norm~$\lambda(\x_t)$, while the randomized strategy achieves the corresponding guarantee with high probability and introduces an additional logarithmic dependence on the failure probability~$\delta$.
However, the greedy strategy requires extra computation to identify the top-$k$ diagonal entries of the matrix $\tilde{\G}_t-\nabla^2 f(\x_{t+1})$ at each iteration, e.g., line \ref{line:greedy}(b) of Algorithm \ref{alg:SRK}.
In contrast, the randomized strategy is easier to implement since it only needs construct $\mU_t$ by following the normal distribution, which can make it preferable in large-scale applications.
\end{remark}
}
We can also set $k=d$ for \srk~methods, which leads to $\eta_t=1$ almost surely for all $t\geq 1$ and achieves the quadratic convergence rate like standard Newton methods. 
\begin{corollary}
\label{cor:recoverNewton}
Under Assumptions~\ref{ass:smooth}, \ref{ass:strongconvex}, and \ref{ass:strongself}, we run the SR-$k$ method (Algorithm~\ref{alg:SRK}) with $k=d$ and $\vx_0\in\BR^d$ such that~$M\lambda(\x_0)\leq 2$,
then $\lambda(\x_{t+1})\leq M(\lambda(\x_t))^2$ is held almost surely 
for all  $t\geq 1$. 
\end{corollary}
\begin{proof}
    The update rule  $\G_{t}=\srk(\tilde{\G}_{t-1},\nabla^2 f(\x_{t}),\U_{t-1})$ implies the equality holds with 
    $\U_{t-1}^{\top}\G_{t}\U_{t-1} = \U_{t-1}^{\top}\nabla^2 f(\x_{t})\U_{t-1} $.
    Our choice of $\U_t\in\RB^{d\times d}$ guarantees that it is non-singular almost surely, so we have $  \G_{t}=\nabla^2 f(\x_{t})$
    for all $t\geq 1$. 
    Lemma~\ref{lm:linear-quadra} with~$\eta_{t}=1$ implies that 
    $ \lambda(\x_{t+1})\leq {M(\lambda(\x_t))^2}/{2} + {M^2(\lambda(\x_t))^3}/{2} \leq M(\lambda(\x_t))^2$
    %
    holds for all $t\geq 1$ almost surely, which finishes the proof.
\end{proof}

\section{Improved Results for Block BFGS and DFP Methods}\label{sec:Block BFGS}
Following our our investigation on \srk~methods, 
we can also achieve the non-asymptotic superlinear convergence rates of randomized block BFGS and randomized block DFP methods \citep{gower2016stochastic,gower2017randomized}.
The block BFGS~\citep{schnabel1983quasi,gower2017randomized,gower2016stochastic} and DFP~\citep{gower2017randomized} updates are defined as follows.

\begin{definition}
Let $\A\in\RB^{d\times d}$ and $\G\in \RB^{d\times d}$ be two positive-definite symmetric matrices with $\A\preceq \G$. For any full rank matrix $\U\in\RB^{d\times k}$ with $k\leq d$, we define 
\begin{align}\label{update:RaBFGS}
   {\text{\rm BlockBFGS}}(\G,\A,\U) 
   \triangleq \G\!-\G\U\big(\U^{\top}\G\U\big)^{-1}\U^{\top}\G+\A\U\big(\U^{\top}\A\U\big)^{-1}\U^{\top}\A,
\end{align}
and
\begin{align}
\label{eq:dfp-update}
\begin{split}
    &{\text{\rm BlockDFP}}(\G,\A,\U)\\
    &\triangleq \A\U\big(\U^{\top}\A\U\big)^{-1}\U^{\top}\A+\big(\I_d-\A\U\big(\U^{\top}\A\U\big)^{-1}\U^{\top}\big)\G\big(\I_d-\U\big(\U^{\top}\A\U\big)^{-1}\U^{\top}\A\big).
  \end{split}
\end{align}
\end{definition}
In previous work, \cite{gower2016stochastic} and \cite{kovalev2020fast} proposed a randomized block BFGS method 
 by constructing the Hessian estimator with formula (\ref{update:RaBFGS}) and showed it has an asymptotic local superlinear convergence rate. 
On the other hand, \cite{gower2017randomized} studied the randomized block DFP update (\ref{eq:dfp-update}) for matrix approximation, but they did not consider solving optimization problems.

To achieve the explicit superlinear convergence rate of block BFGS and block DFP methods, we provide some properties of block BFGS and block DFP updates which are similar to our observations on \srk~update.
We first show that block BFGS and DFP updates also have the non-increasing deviation from the target matrix.
\begin{lemma}
\label{lm:bfgsnofar}
Given any positive-definite matrices $\mA\in\BR^{d\times d}$ and $\mG\in\BR^{d\times d}$ which satisfy $\A\preceq\G\preceq\eta \A$
for some~$\eta\geq 1$, both of the updates $\G_{+}={ \text{\rm BlockBFGS}}(\G,\A,\U)$ and $\G_{+}={ \text{\rm BlockDFP}}(\G,\A,\U)$ for full rank matrix~$\U\in\RB^{d\times k}$ hold that $ \A\preceq\G_{+}\preceq \eta\A. $
\end{lemma}
\begin{proof}
We first consider the block BFGS update $\G_{+}={ \text{\rm BlockBFGS}}(\G,\A,\U)$. 
According to the Woodbury identity~\citep{woodbury1950inverting}, we have
\begin{align}\label{update:invRaBFGS}
\begin{split}    
\!\!\G_{+}^{-1} = \U\big(\U^{\top}\A\U\big)^{-1}\U^{\top} \!+\! \big(\I_d-\U\big(\U^{\top}\A\U\big)^{-1}\U^{\top}\A\big)\G^{-1}\big(\I_d-\A\U\big(\U^{\top}\A\U\big)^{-1}\U^{\top}\big).
\end{split}
\end{align}
The condition $\A\preceq\G\preceq\eta \A$ means $\eta^{-1}\A^{-1}\preceq\G^{-1}\preceq \A^{-1}$,
which implies
\begin{align*}
    \G_{+}^{-1}\!\overset{\eqref{update:invRaBFGS}}{\preceq} & \U\big(\U^{\top}\A\U\big)^{-1}\U^{\top} \! +\! \big(\I_d-\!\U\big(\U^{\top}\A\U\big)^{-1}\U^{\top}\A\big)\A^{-1}\big(\I_d-\A\U\big(\U^{\top}\A\U\big)^{-1}\U^{\top}\big) \!=\! \A^{-1}
\end{align*}
and
\begin{align*}
\G_{+}^{-1}\!\!
&\overset{\eqref{update:invRaBFGS}}{\succeq}\!\!\U\big(\U^{\top}\!\A\U\big)^{-1}\U^{\top}\!\!+\!\eta^{-1}\!\big(\I_d\!-\!\U\big(\U^{\top}\!\A\U\big)^{-1}\U^{\top}\!\A\big)\!\A\!^{-1}\!\big(\I_d\!-\!\A\U\big(\U^{\top}\!\A\U\big)^{-1}\U^{\top}\big)\\
    &~=\eta^{-1}\A ^{-1}+(1-\eta^{-1})\U\big(\U^{\top}\A\U\big)^{-1}\U^{\top}\succeq \eta^{-1}\A^{-1}.
\end{align*}
Thus we have $ {\eta}^{-1}\A^{-1}\preceq\G^{-1}_{+}\preceq \A^{-1}$, which finishes the proof for the block BFGS update.
We then consider the DFP update $\G_{+}={\text{\rm BlockDFP}}(\G,\A,\U)$.
The condition $\A\preceq\G\preceq\eta \A$ means
\begin{align*}
\G_{+}\!\!\!\overset{\eqref{eq:dfp-update}}{\succeq}
\A\U\big(\U^{\top}\!\A\U\big)^{-1}\U^{\top}\A +\!\big(\I_d\!-\!\A\U\big(\U^{\top}\!\A\U\big)\!^{-1}\U^{\top}\big)\A\big(\I_d\!-\!\U\big(\U^{\top}\!\A\U\big)^{-1}\U^{\top}\!\A\big)\!=\!\A
\end{align*}
and
\begin{align*}
 \G_{+}&\overset{\eqref{eq:dfp-update}}{\preceq}
\A\U\big(\U^{\top}\A\U\big)^{-1}\U^{\top}\A +\eta\big(\I_d-\A\U\big(\U^{\top}\A\U\big)^{-1}\U^{\top}\big)\A\big(\I_d-\U\big(\U^{\top}\A\U\big)^{-1}\U^{\top}\A\big)\\
&\,\,\,= \eta\A +(1-\eta)\A\U\big(\U^{\top}\A\U\big)^{-1}\U^{\top}\A \preceq \eta\A,
\end{align*}
where the last inequality is due to the facts $\eta\geq 1$ and $\A\U\big(\U^{\top}\A\U\big)^{-1}\U^{\top}\A\succeq\0$, which finishes the proof for the block DFP update.
\end{proof}

Then we introduce the following quantity~\citep{rodomanov2021greedy}
\begin{align}
\label{eq:measurebfgs}
    \sigma_{\A}(\G)\triangleq\tr{\A^{-1}(\G-\A)}
\end{align}
to measure the difference between the target matrix $\mA\in\BR^{d\times d}$ and the current estimator~$\mG\in\BR^{d\times d}$. 
In the following theorem, we show that randomized block BFGS and DFP updates converge to the target matrix with a faster rate than the ordinary randomized BFGS and DFP updates~\citep{rodomanov2021greedy,lin2021greedy}.
\begin{theorem}\label{thm:bfgs}
Consider the randomized block update
\begin{align}\label{eq:bfgsupdate}
    \G_{+}={\text{\rm BlockBFGS}}(\G,\A,\U)\qquad\text{or}\qquad\G_{+}={\text{\rm BlockDFP}}(\G,\A,\U),
\end{align}
where $\G,\A\in\RB^{d\times d}$ and $\G\succeq\A$. If
$\U\in\BR^{d\times k}$ is chosen as~$[\mU]_{ij} \overset{{\rm i.i.d}}{\sim} \fN(0,1)$ and it satisfies that $\mu\I_d\preceq\A\preceq L\I_d$, then we have
\begin{align}
\label{eq:bfgssigma}
  \EB\left[\sigma_{\A}(\G_{+})\right]\leq \left(1-\frac{k}{d\varkappa}\right)\sigma_\A(\G).
\end{align}
\end{theorem}

\begin{proof}
The condition $\mu\I_d\preceq\A\preceq L\I_d$ means we have
\begin{align}
    \label{eq:Aneq}   \big(\U^{\top}\A\U\big)^{-1} \succeq \frac{1}{L}\cdot\big(\U^{\top}\U\big)^{-1} \qquad\text{and}\qquad \frac{1}{\mu}\cdot\I_d-\A^{-1}\succeq\0,
\end{align}
which leads to the following inequality
\begin{align}
\label{eq:trineq2}
\begin{split}           
    &{\rm tr}\big(\!\big(\U^{\top}\A\U\big)^{-1}\U^{\top}\G\U\big)-{\rm tr}\big(\A\U\big(\U^{\top}\A\U\big)^{-1}\U^{\top}\big)\\
    &={\rm tr}\big(\!\big(\U^{\top}\A\U\big)^{-1}\U^{\top}(\G-\A)\U\big)
    \overset{\eqref{eq:Aneq}}{\geq} \frac{1}{L}\tr{\big(\U^{\top}\U\big)^{-1}\U^{\top}(\G-\A)\U}
  \\
  &=\frac{1}{L}\tr{{(\G-\A)^{1/2}\U\big(\U^{\top}\U\big)^{-1}\U^{\top}(\G-\A)^{1/2}}}\\
  &\!\overset{\eqref{eq:Aneq}}{\geq} \frac{\mu}{L} \tr{\A^{-1}(\G-\A)^{1/2}\U\big(\U^{\top}\U\big)^{-1}\U^{\top}(\G-\A)^{1/2}}.
\end{split}
 \end{align}
For the block BFGS update, we set $\P=\A^{1/2}\U$. Then it holds
\begin{align}\label{eq:bfgs-proof-psd}
\begin{split}
  &\U^{\top}\G{\big(\A^{-1}-\U\big(\U^{\top}\A\U\big)^{-1}\U^{\top}\big)}\G\U\\
  &= \U^{\top}\G\A^{-1/2}\big(\I_d -\A^{1/2}\U\big(\U^{\top}\A\U\big)^{-1}\U^{\top}\A^{1/2}\big)\A^{-1/2}\G\U\\
    &= \U^{\top}\G\A^{-1/2}\big(\I_d-\P\big(\P^{\top}\P\big)^{-1}\P^{\top}\big)\A^{-1/2}\G\U
    \succeq \0,
\end{split}
\end{align}
which implies 
\begin{align}
\label{eq:trineq}
\begin{split}               
& \tr{\U\G\big(\U^{\top}\G\U\big)^{-1}\U^{\top}\G\A^{-1}}-\tr{\big(\U^{\top}\A\U\big)^{-1}\U^{\top}\G\U} \\
&=\tr{\big(\U^{\top}\G\U\big)^{-1}\U^{\top}\G\A^{-1}\G\U}-\tr{\big(\U^{\top}\A\U\big)^{-1}\U^{\top}\G\U}\\    &=\tr{\big(\U^{\top}\G\U\big)^{-1}\big(\U^{\top}\G\A^{-1}\G\U - \U^{\top}\G\U\big(\U^{\top}\A\U\big)^{-1}\U^{\top}\G\U\big)}\\    &=\text{tr}\big(\big(\U^{\top}\G\U\big)\!^{-1/2}\U^{\top}\G{\big(\A^{-1}-\U\big(\U^{\top}\A\U\big)^{-1}\U^{\top}\big)}\G\U\big(\U^{\top}\G\U\big)\!^{-1/2}\big)\overset{\eqref{eq:bfgs-proof-psd}}{\geq}\! 0.
\end{split}
\end{align}
The equation \eqref{update:RaBFGS} implies
\begin{align*}
&\tr{(\G_{+}-\A)\A^{-1}} \\ &\overset{\eqref{update:RaBFGS}}{=}\tr{(\G-\A)\A^{-1}}-\tr{\G\U\big(\U^{\top}\G\U\big)^{-1}\U^{\top}\G\A^{-1}}+\tr{\A\U\big(\U^{\top}\A\U\big)^{-1}\U^{\top}}\\
 &\overset{\eqref{eq:trineq}}{\leq} \tr{(\G-\A)\A^{-1}}- \big(\tr{\big(\U^{\top}\A\U\big)^{-1}\U^{\top}\G\U}-\tr{\A\U\big(\U^{\top}\A\U\big)^{-1}\U^{\top}}\big).
\end{align*}
Taking expectation on both sides of the above result, we obtain 
\begin{align}
\label{eq:sigmaupdate}
\begin{split}
& \EBcommon{\sigma_{\A}(\G_{+})}
 \overset{\eqref{eq:measurebfgs}}{\leq} \sigma_{\A}(\G)-\EBP{\tr{\big(\U^{\top}\A\U\big)^{-1}\U^{\top}\G\U-\A\U\big(\U^{\top}\A\U\big)^{-1}\U^{\top}}} \\
   &\overset{\eqref{eq:trineq2}}{\leq}\sigma_{\A}(\G) - \BE\left[\frac{\mu}{L}\tr{\A^{-1}(\G-\A)^{1/2}\U\big(\U^{\top}\U\big)^{-1}\U^{\top}(\G-\A)^{1/2}}\right]\\
   &~ = \sigma_{\A}(\G) - \frac{\mu}{L}\EBP{\tr{(\G-\A)^{1/2}\A^{-1}(\G-\A)^{1/2}\U(\U^{\top}\U)^{-1}\U^{\top}}}\\
   &~ = \sigma_{\A}(\G)- \frac{k}{d\varkappa} \tr{(\G-\A)^{1/2}\A^{-1}(\G-\A)^{1/2}} = \left(1-\frac{k}{d\varkappa}\right)\sigma_\A(\G),
\end{split}
\end{align}
where the last line is by applying  Lemma~\ref{lm:explicitbound}(a) with $\R=(\G-\A)^{1/2}\A^{-1}(\G-\A)^{1/2}$. 

For the block DFP update, the equation \eqref{eq:dfp-update} indicates 
\begin{align*}
\begin{split}
        & \tr{ (\G_{+}-\A)\A^{-1}}\\
   &\overset{\eqref{eq:dfp-update}}{=} \tr{\A\U(\U^{\top}\A\U)^{-1}\U^{\top}}\!+\tr{(\G-\A)\A^{-1}}\!+\tr{\A\U(\U^{\top}\A\U)^{-1}\U^{\top}\G\U(\U^{\top}\A\U)^{-1}\U^{\top}} \\
   &~~~~~~ -\tr{\A\U(\U^{\top}\A\U)^{-1}\U^{\top}\G\mA^{-1}}-\tr{\G\U(\U^{\top}\A\U)^{-1}\U^{\top}}\\
   &~=~\tr{(\G-\A)\A^{-1}} - \big(\tr{(\U^{\top}\A\U)^{-1}(\U^{\top}\G\U)}-\tr{\A\U\big(\U^{\top}\A\U\big)^{-1}\U^{\top}}\big).
\end{split}   
\end{align*}
Taking the expectation on both sides of the above equation, then we can follow the analysis of the block BFGS update from the step of~\eqref{eq:sigmaupdate} to prove the desired result for the block DFP update. Hence, we finish the proof.
\end{proof}
\begin{remark}   
Notice that \cite{gower2017randomized} also established convergence rates for randomized block BFGS and DFP updates, but their results do not reveal the relationship between the convergence rates and the block size $k$.
\end{remark}

\begin{algorithm}[ht]
\caption{Randomized Block BFGS/DFP}\label{alg:bfgs}
\begin{algorithmic}[1]
\STATE \textbf{Input:} $\x_0$, $\G_0$, and $k$ \\
\STATE \textbf{for} $t=0,1\dots$\\
\STATE \quad $\x_{t+1}=\x_t-\G_t^{-1}\nabla f(\x_t)$ \\
\STATE \quad $r_t=\|\x_{t+1}-\x_{t}\|_{\x_t}$ \\[0.08cm]
\STATE \quad $\tilde{\G}_{t}=(1+Mr_t)\G_t$ \label{line:correction} \\
\STATE \quad Construct $\U_t$ by $\left[\U_{t}\right]_{ij}\overset{\rm{i.i.d}}{\sim} {\fN(0,1)}$ \\[0.03cm]
\STATE \quad Update $\G_t$ by  \\[0.05cm]
\quad \quad  (a) $\G_{t+1}=\bfgs(\tilde{\G}_t,\nabla^2f(\x_{t+1}),\U_t)$\\[0.05cm]
\quad \quad (b) $\G_{t+1}= \dfp(\tilde{\G}_t,\nabla^2f(\x_{t+1}),\U_t)$ \\
\STATE \textbf{end for} 
\end{algorithmic}
\end{algorithm}

We propose randomized block BFGS and DFP methods in Algorithm~\ref{alg:bfgs}.
Using the results of Theorem~\ref{thm:bfgs}, we establish the explicit superlinear convergence rate for our methods as follows  (see Appendix \ref{sec:bfgs_proof} for the detailed proof).

\begin{theorem}
\label{thm:BFGS}
Under Assumption~\ref{ass:smooth}, \ref{ass:strongconvex}, and \ref{ass:strongself}, we run randomized block BFGS/DFP method (Algorithm~\ref{alg:bfgs}) with $\vx_0\in\BR^d$ and $\mG_0\in\BR^{d\times d}$ such that
\begin{align}
\label{eq:bfgsini}
 M \lambda(\x_0)\leq \frac{\ln 2}{4}\cdot \frac{d\varkappa - k}{\eta_0 d^2\varkappa} 
\qquad\text{and}\qquad
\nabla^2 f(\x_0)\preceq\G_0\preceq \eta_0\nabla^2f(\x_0)
\end{align}
for some $k<d$ and $\eta_0\geq 1$. 
Then we have 
\begin{align*}
    \BE\left[\frac{\lambda(\x_{t+1})}{\lambda(\x_t)}\right]\leq 2d\eta_0\left(1-\frac{k}{d\varkappa}\right)^{t}.
\end{align*}
\end{theorem}


For randomized block BFGS method, we can introduce the scaled directions to eliminate the condition number in the results of Theorem \ref{thm:bfgs} and \ref{thm:BFGS}. 
We first improve the result for matrix approximation as follows.

\begin{theorem}\label{thm:bfgs-1}
Consider the block BFGS update
\begin{align}
\label{eq:bfgsupdate-1}
    \G_{+}={ \text{\rm BlockBFGS}}(\G,\A,\LL^{\top}{\U})
\end{align}
where $\mG\succeq\A\in\RB^{d\times d}$.
If $\mu\I\preceq\A\preceq L\I$,  $\LL\in\BR^{d\times d}$ satisfies that $\LL^{\top}\LL=\G^{-1}$, and $\U\in\BR^{d\times k}$ is chosen as $[\mU]_{ij} \overset{{\rm i.i.d}}{\sim} \fN(0,1)$.
Then we have
\begin{align}
\label{eq:fasterbfgssigma-1}
    \EB\left[\sigma_{\A}(\G_{+})\right]\leq \left(1-\frac{k}{d}\right)\sigma_\A(\G).
\end{align}
\end{theorem}
\begin{proof}
According to the block-BFGS update \eqref{update:RaBFGS}, we have
\begin{align*}
   &\tr{ (\G_{+}-\A)\A^{-1}}\\
   &=\tr{(\G-\A)\A^{-1}}-\tr{\LL^{-1}\U\big(\U^{\top}\U\big)^{-1}\U^{\top}\LL^{-\top}\A^{-1}}+\tr{\A\LL^{\top}\U\big(\U\LL^{\top}\A\LL^{\top}\U\big)^{-1}\U^{\top}\LL}\\
   &=\tr{(\G-\A)\A^{-1}}-\tr{\LL^{-1}\U\big(\U^{\top}\U\big)^{-1}\U^{\top}\LL^{-\top}\A^{-1}}+k.
\end{align*}
The second term on the right-hand side holds that
\begin{align*}   
\EBP{\tr{\LL^{-1}\U\big(\U^{\top}\U\big)^{-1}\U^{\top}\LL^{-\top}\A^{-1}}}\! = \tr{\EBP{\U(\U^{\top}\U)^{-1}\U^{\top}}\LL^{-\top}\A^{-1}\LL^{-1}}
\!\overset{\eqref{eq:EP}}{=}\frac{k}{d}\tr{\A^{-1}\G}.
\end{align*} 
Thus, we combine the above results and get
\begin{align*}
 & \EB\left[\sigma_{\A}(\G_{+})\right] 
   =\sigma_\A(\G)-\frac{k}{d}\tr{\A^{-1}\G}+k  =\!\sigma_\A(\G)-\frac{k}{d}\tr{\A^{-1}(\G-\A)} 
   =\!\left(1-\frac{k}{d}\right)\sigma_\A(\G).
\end{align*}
\end{proof}
For given matrices $\mG,\mA,\mL,\mU\in\BR^{d\times d}$ satisfying the condition of Theorem \ref{thm:bfgs-1}, we can access $\LL_+\in\BR^{d\times d}$ such that $\G_{+}^{-1}=\LL
_{+}^\top\LL_{+}$ by a closed form expression~\cite[Section 9]{gower2017randomized}.

\begin{proposition}
\label{prop:efficientL}
Under the setting of Theorem \ref{thm:bfgs-1}, the matrix 
\begin{align*}\begin{split}
    \LL_{+} &= {\text{\rm UpdateL}}(\LL,\A,\U)\\
    & \triangleq\LL +\big(\U\big(\U^{\top}\U\big)^{-1/2}- \LL\A\LL^{\top}\U\big(\U^{\top}\LL\A\LL^{\top}\U\big)^{-1/2}\big)\big(\U^{\top}\LL\A\LL^{\top}\U\big)^{-1/2}\U^{\top}\LL,
\end{split}
\end{align*}
holds that $\G_{+}^{-1}=\LL_{+}^{\top}\LL_+$.
\end{proposition}
We can achieve Proposition \ref{prop:efficientL} by applying equation (9.11) of \cite{gower2017randomized} with~$\tilde{\S}_k=\U$ and $\LL_k=\LL^{\top}$, 
where $\tilde{\S}_k$ and $\mL_k$ follow notations of \cite{gower2017randomized}.
This result means we can access the scaling matrix $\mL_+$ efficiently during iterations.

Based on the BFGS update with scaling directions shown in Theorem \ref{thm:bfgs-1} and the efficient update rule shown in Proposition~\ref{prop:efficientL}, we propose the faster randomized block BFGS method in Algorithm~\ref{alg:fasterbfgs},
which has the following local convergence rate matching \srk~methods (see Appendix~\ref{app:fasterBFGSproof} for the detailed proof).
\begin{algorithm}[ht]
\caption{Faster Randomized Block BFGS}\label{alg:fasterbfgs}
\begin{algorithmic}[1]
\STATE \textbf{Input:} $\x_0$, $\G_0$, $\mL_0$, and $k$ \\
\STATE \textbf{for} $t=0,1\dots$\\
\STATE \quad $\x_{t+1}=\x_t-\G_t^{-1}\nabla f(\x_t)$ \\
\STATE \quad $r_t=\|\x_{t+1}-\x_{t}\|_{\x_t}$ \\[0.05cm]
\STATE \quad $\tilde{\G}_{t}=(1+Mr_t)\G_t$ \\[0.05cm]
\STATE \quad $\tilde{\LL}_{t}=\LL_t/\sqrt{1+Mr_t}$ \\
\STATE \quad Construct $\U_t$ by $\left[\U_{t}\right]_{ij}\overset{\rm{i.i.d}}{\sim} {\fN(0,1)}$ \\[0.05cm]
\STATE \quad $\G_{t+1}= \bfgs(\tilde{\G}_t,\nabla^2f(\x_{t+1}),\tilde{\LL}_{t}^{\top}\U_t)$ \label{line:fastBFGS} \\[0.05cm]
\STATE \quad ${\LL}_{t+1} = {\text{\rm UpdateL}}(\tilde{\LL}_t,\nabla^2 f(\x_{t+1}), \U_t)$
\STATE \textbf{end for}
\end{algorithmic}
\end{algorithm} 
\begin{theorem}
\label{thm:fasterBFGS}
Under Assumption~\ref{ass:smooth}, \ref{ass:strongconvex}, and \ref{ass:strongself}, we run faster randomized block BFGS method (Algorithm~\ref{alg:fasterbfgs}) with $\vx_0\in\BR^d$ and $\mL_0,\mG_0\in\BR^{d\times d}$ such that
\begin{align}
\label{eq:fasterbfgsini}
M\lambda_0 \leq \frac{\ln 2}{4}\cdot \frac{d-k}{\eta_0 d^2},  \quad
\nabla^2 f(\x_0)\preceq\G_0\preceq \eta_0\nabla^2f(\x_0), \quad \text{and} \quad \G_0^{-1}=\LL_0^{\top}\LL_0
\end{align}
for some $k<d$ and $\eta_0\geq 1$. 
Then we have 
\begin{align*}
    \BE\left[\frac{\lambda(\x_{t+1})}{\lambda(\x_t)}\right]\leq 2d\eta_0\left(1-\frac{k}{d}\right)^{t}.
\end{align*}
\end{theorem}

\begin{remark}
The condition on $\mG_0\in\BR^{d \times d}$ in Theorem~\ref{thm:BFGS} and \ref{thm:fasterBFGS} can be satisfied by simply setting $\mG_0=L\mI_d$. 
In addition, we can set $\mL_0=L^{-1/2}\mI_d$ to satisfy the condition on $\mL_0\in\BR^{d \times d}$ in Theorem \ref{thm:fasterBFGS}, then we can efficiently implement the update on $\LL_t$ by Proposition~\ref{prop:efficientL}.
\end{remark}

\begin{remark}
For $k=1$, the convergence rates provided by Theorem~\ref{thm:BFGS}~and~\ref{thm:fasterBFGS} match the results of the ordinary randomized BFGS/DFP methods~\citep{rodomanov2021greedy,lin2021greedy} and fast randomized BFGS methods~\citep{lin2021greedy} respectively. 
For $k=d$, the updates \eqref{eq:bfgsupdate} and \eqref{eq:bfgsupdate-1} holds $\mG_+=\mA$ almost surely. This leads to the local quadratic convergence rate of our (faster) randomized block BFGS/DFP methods, which is similar to the behaviors of \srk~methods shown in Corollary \ref{cor:recoverNewton}. 
\end{remark}



\section{Numerical Experiments}\label{sec:exp}

We conduct the experiments on the model of regularized logistic regression, which can be formulated as 
\begin{align}
\label{prob:logstic}
    \min_{\vx\in\BR^d} f(\x)\triangleq \frac{1}{n}\sum_{i=1}^{n}\ln\left(1+\exp{-b_i\va_i^{\top}\x}\right)+\frac{\gamma}{2}\|\x\|^2,
\end{align}
where $\va_i\in\RB^d$ and $b_i\in\{-1,+1\}$ are the feature and the corresponding label of the $i$-th sample respectively, and $\gamma>0$ is the regularization hyperparameter.

We refer to \srk~methods with randomized/greedy strategies (Algorithm~\ref{alg:SRK})~as R-\srk/G-\srk.
The SR1 methods with randomized/greedy strategies are referred to as R-SR1/G-SR1 \cite[Algorithm 4]{lin2021greedy}.
We refer to the existing randomized block BFGS method \cite[Algorithm 1]{kovalev2020fast} as RB-BFGSv1 and our randomized block BFGS/DFP methods (Algorithm~\ref{alg:bfgs}) as RB-BFGSv2/RB-DFP.
We denote our faster block BFGS method (Algorithm~\ref{alg:fasterbfgs}) as
FRB-BFGS.
We compare the proposed R-\srk, G-\srk, RB-BFGSv2, RB-DFP, and FRB-BFGS with baselines on problem~(\ref{prob:logstic}).
We do not include the empirical results of classical SR1/BFGS/DFP methods since G-SR1 has competitive performance with them \cite{rodomanov2021greedy}.
For all methods, we tune $\G_0$ and $M$ from~$\{\I_d, 10\I_d, 10^2\I_d, 10^3\I_d, 10^4\I_d\}$ and $\{1,10,10^2,10^3,10^4\}$ respectively.
We evaluate the performance on datasets ``MNIST'' ($d= 780$), ``sido0'' ($d=4,932$), and ``gisette'' ($d=5,000$). 
We conduct experiments using Python 3.8.12 on a PC with Apple M1.

We present the results of ``iteration numbers vs. gradient norm'' and ``running time (seconds) vs. gradient norm'' in Figure \ref{fig:experiment-10}, where we take $k=200$ for block quasi-Newton methods (R-\srk, G-\srk, RB-BFGSv1, RB-BFGSv2, and FRB-BFGS).
We observe the proposed R-\srk~and G-\srk~significantly outperform other methods.
The proposed block BFGS methods (RB-BFGSv2 and FRB-BFGS) outperform the block DFP method (RB-DFP), which is similar to the advantage of the BFGS method over the DFP method~\citep{rodomanov2021greedy}.


We also demonstrate the impact of parameter $k$ for  \srk~methods. 
It is natural that Figure~\ref{fig:diff_k_ra}(a), 2(b), and \ref{fig:diff_k_ra}(c) show the larger $k$ leads to faster convergence in terms of iteration numbers, which validates our theoretical analysis in Section \ref{sec:srk-opt}.
In addition, Figure~\ref{fig:diff_k_ra}(d), 2(e), and \ref{fig:diff_k_ra}(f) show SR-$k$ methods with $k>1$ significantly outperform SR$1$ in terms of the running time. 
This is because the block update can reduce the cache miss rate and take advantage of parallel computing.  
However, increasing~$k$ does not always result in less running time because the acceleration caused by the block update is limited by the cache size.

\begin{figure}[t]
\centering
\begin{tabular}{cccc}
\includegraphics[scale=0.26]{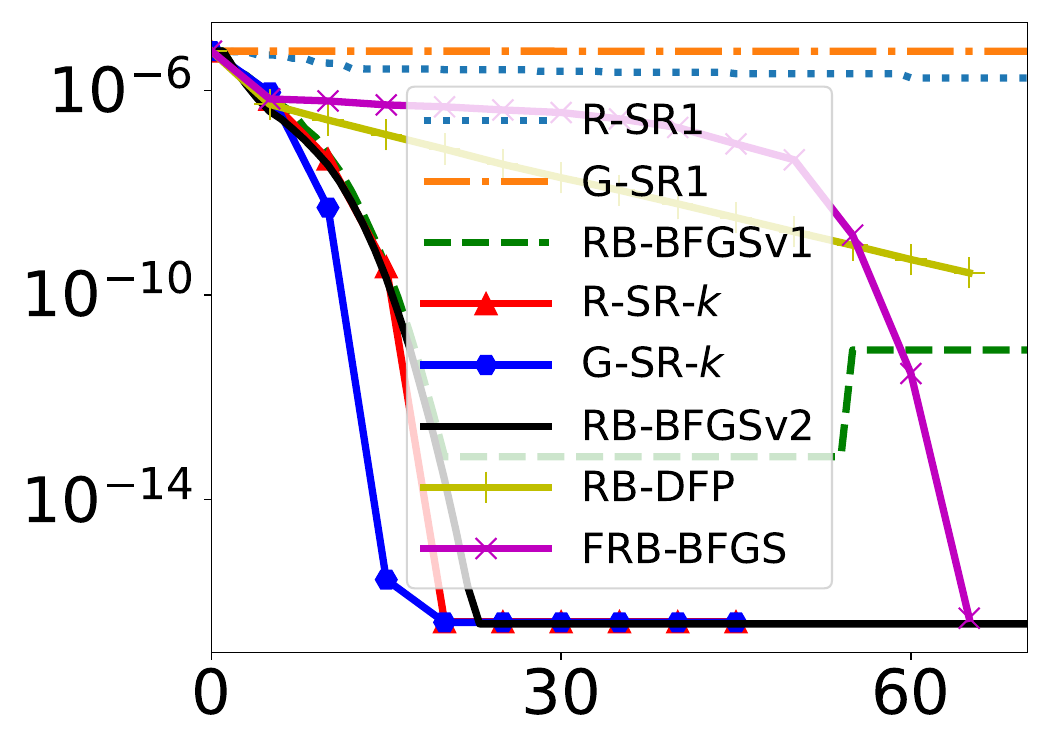} &
\includegraphics[scale=0.26]{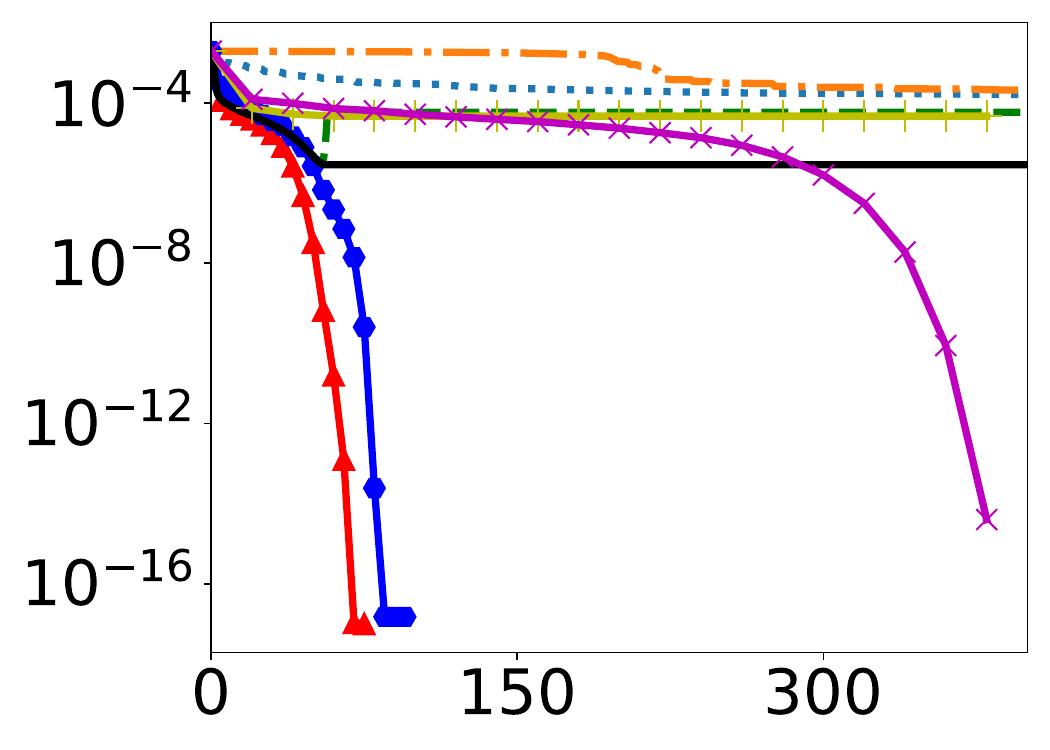} &
\includegraphics[scale=0.26]{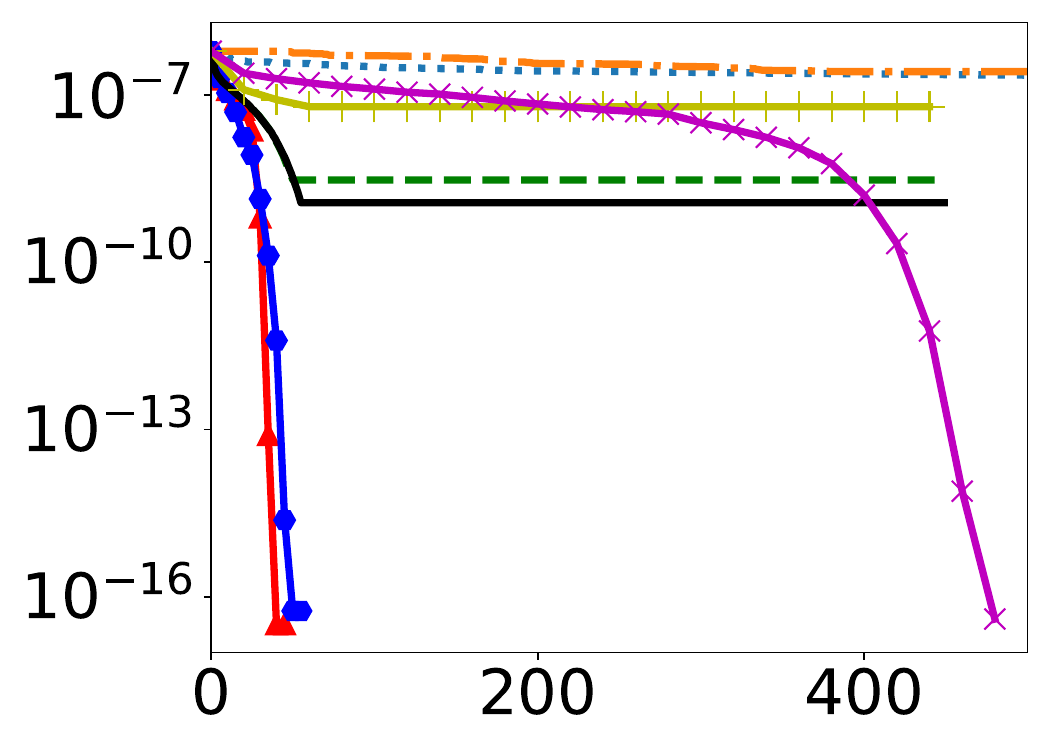}
\\[-0.2cm]
\footnotesize~~~~~ (a) MNIST (iteration) & \footnotesize~~~~~  (b) sido0 (iteration) & \footnotesize~~~~~  (c) gisette (iteration) \\[0.15cm]
\includegraphics[scale=0.26]{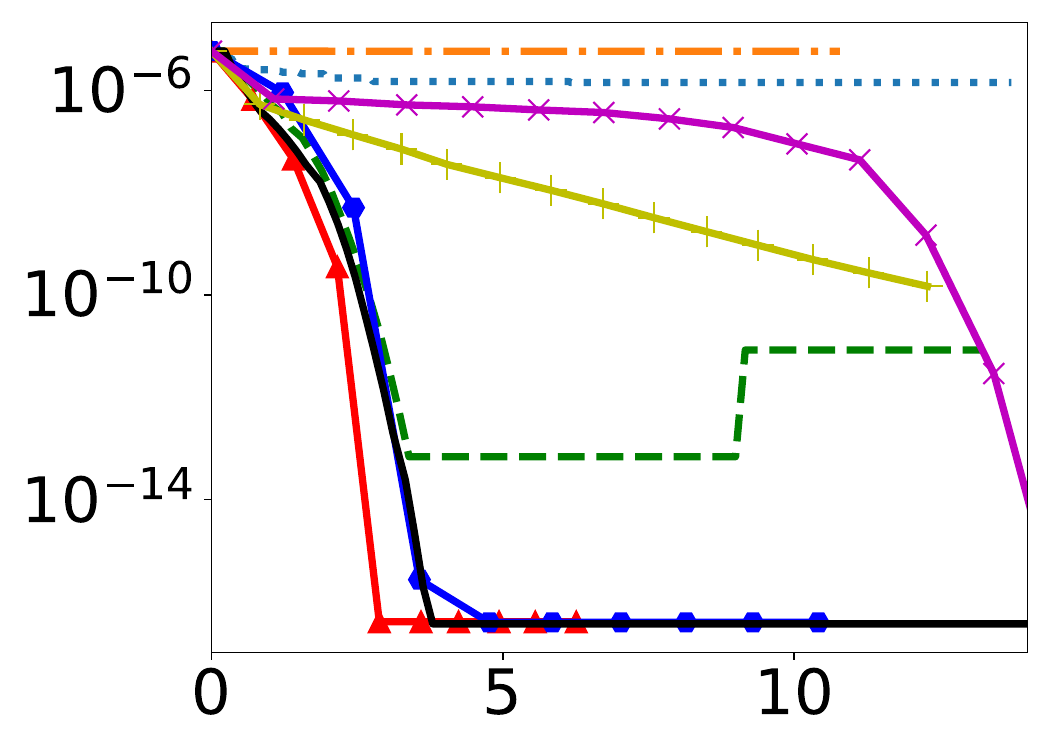} & 
\includegraphics[scale=0.26]{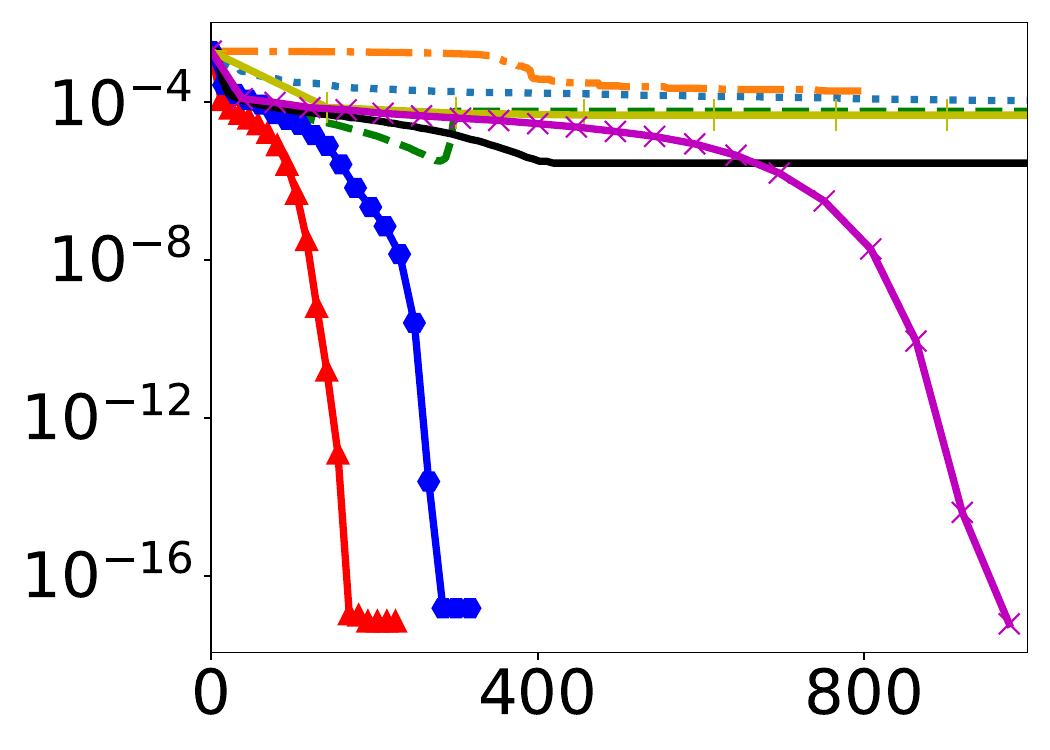} &
\includegraphics[scale=0.26]{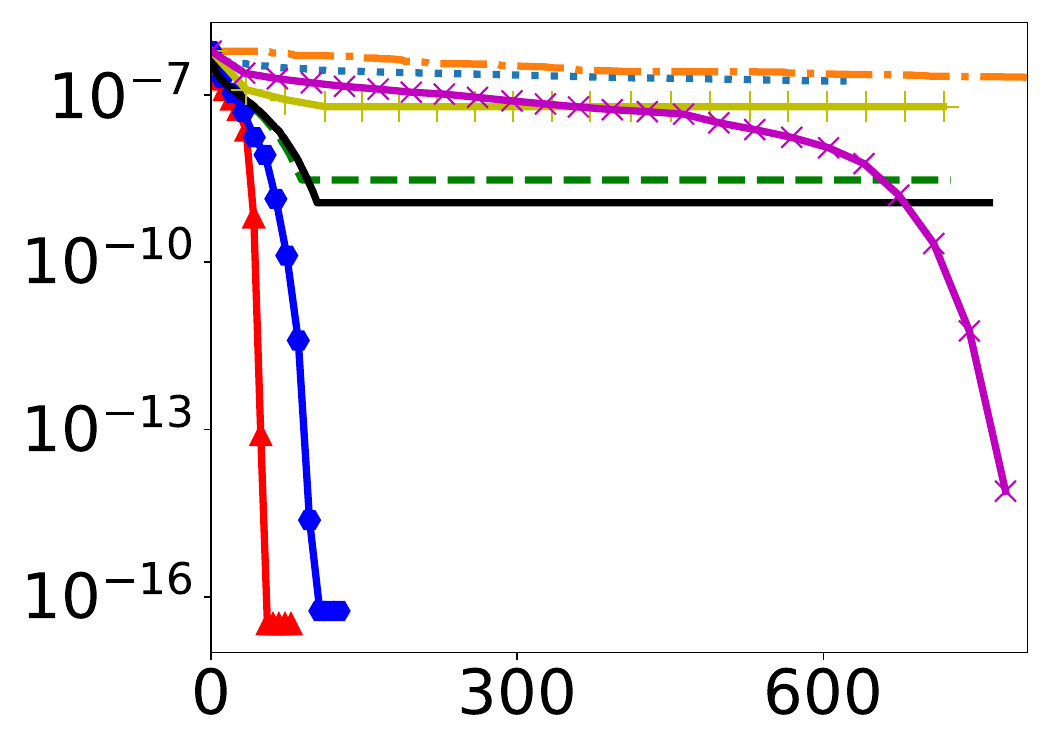}
\\[-0.15cm]
\footnotesize~~~~~   (d) MNIST (time) & \footnotesize~~~~~  (e) sido0 (time) & \footnotesize~~~~~  (f) gisette (time)
\end{tabular}\vskip-0.15cm
\caption{We show ``\#iteration vs. $\|\nabla f(\x)\|$'' and ``running time (s) vs. $\|\nabla f(\x)\|$'' on datasets ``MNIST'', ``sido0'', and ``gisette'', where  we take $k=200$ for all of block quasi-Newton methods.}\label{fig:experiment-10} \vskip-0.4cm
\end{figure}
\begin{figure}[t]
\centering
\begin{tabular}{cccc}
\includegraphics[scale=0.26]{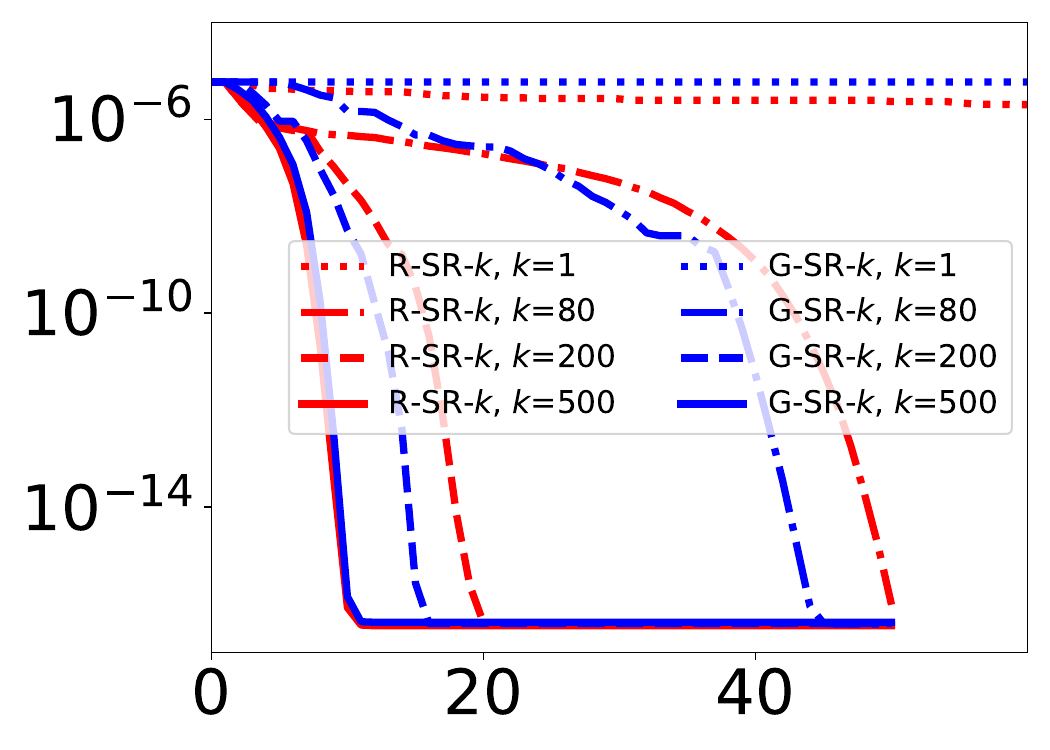} &
\includegraphics[scale=0.26]{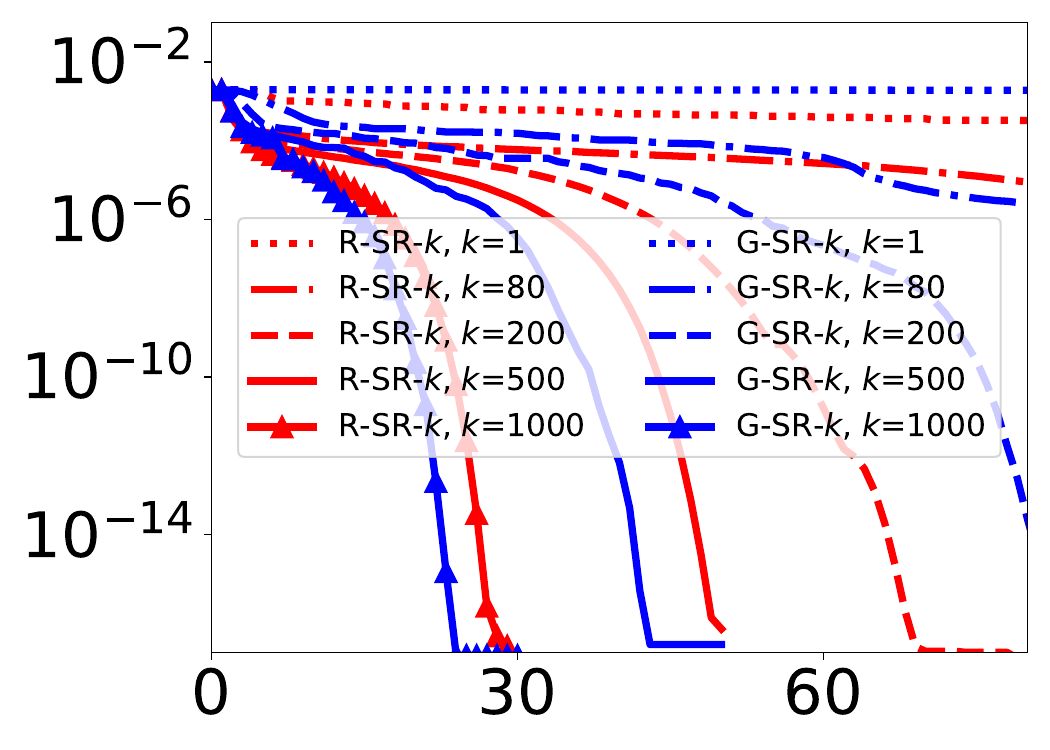} &
\includegraphics[scale=0.26]{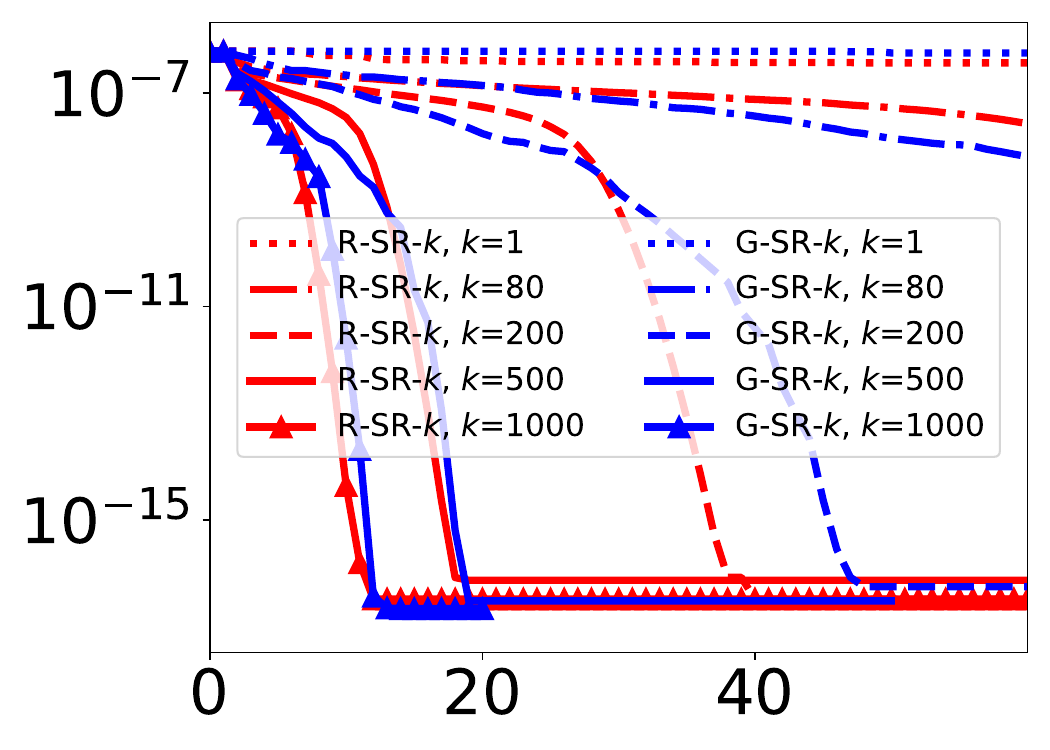}
\\[-0.2cm]
\footnotesize~~~~~ (a) MNIST (iteration) & \footnotesize~~~~~   (b) sido0 (iteration) &\small (c) gisette (iteration) \\[0.15cm]
\includegraphics[scale=0.26]{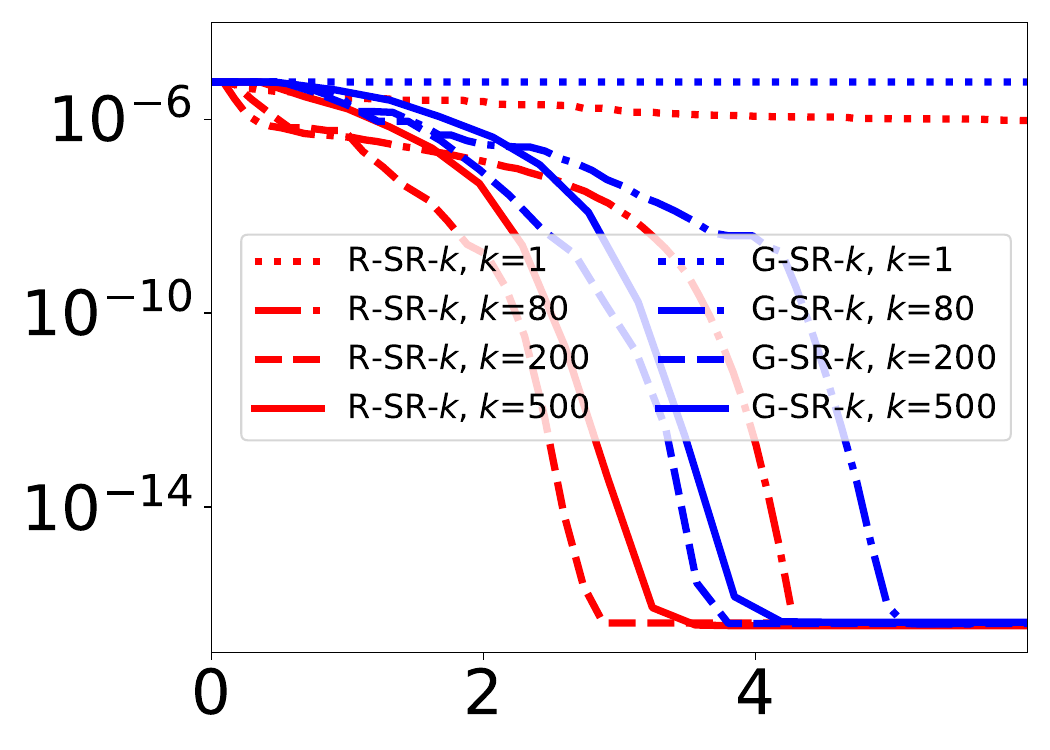} & 
\includegraphics[scale=0.26]{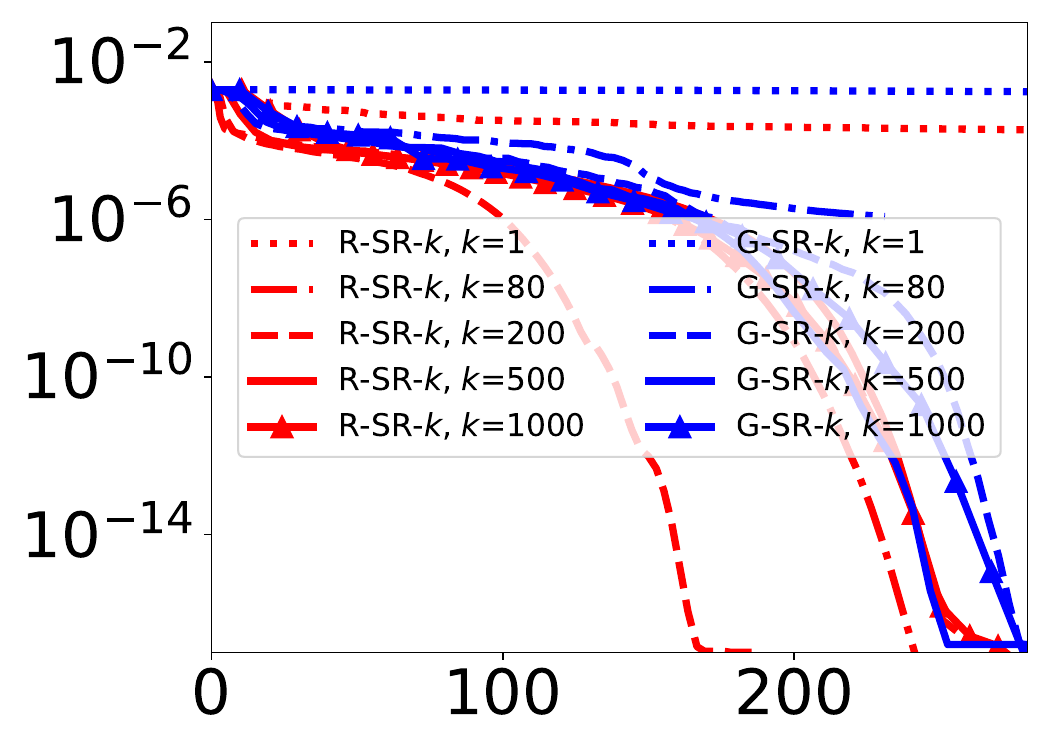} &
\includegraphics[scale=0.26]{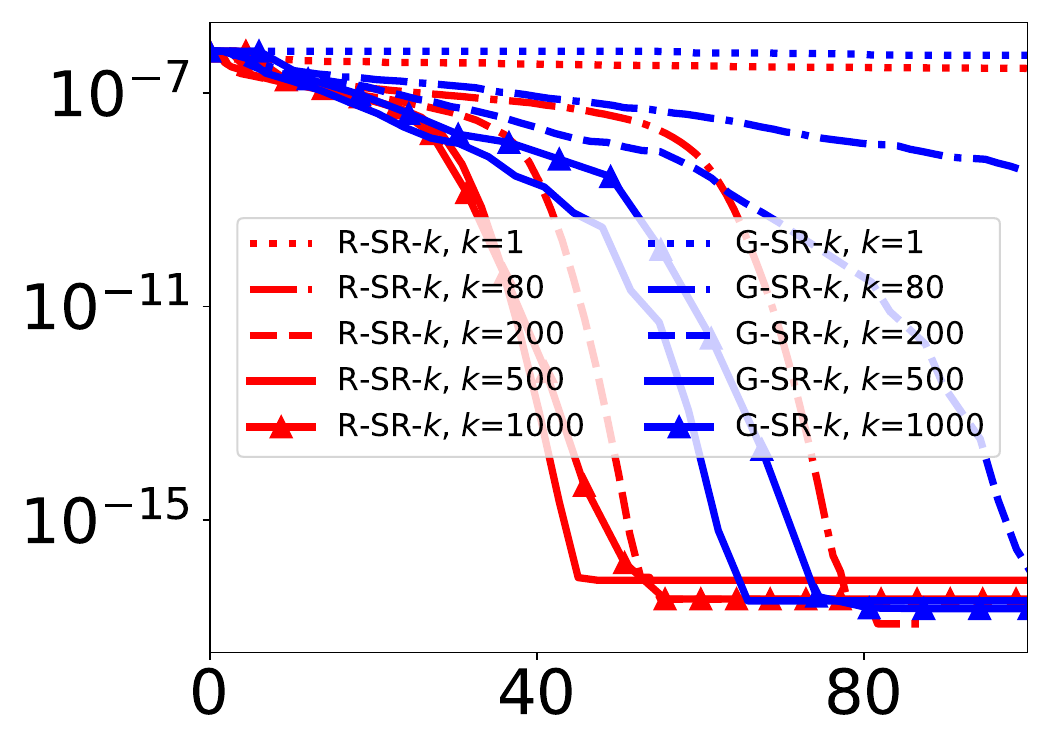}
\\[-0.2cm]
\footnotesize~~~~~   (d) MNIST (time) & \footnotesize~~~~~   (e) sido0 (time) & \footnotesize~~~~~  (f) gisette (time)
\end{tabular}\vskip-0.15cm
\caption{We show ``\#iteration vs. $\|\nabla f(\x)\|$'' and ``running time (s) vs. $\|\nabla f(\x)\|$'' for proposed R-\srk~and G-\srk~with different $k\in\{1,80,200,500\}$ on datasets ``MNIST'', and $k\in\{1,80,200,500,1000\}$ on datasets ``sido0'' and ``gisette''.}\label{fig:diff_k_ra} \vskip-0.5cm
\end{figure}
 {
Figure~\ref{fig:diff_k_ra} further illustrates the practical trade-off between the randomized and greedy strategies.
In most cases, the greedy strategy requires fewer iterations than the randomized strategy, which is consistent with its sharper theoretical guaranty.
However, the greedy strategy may take more CPU time because it needs to compute the top-$k$ diagonal entries of matrix $\tilde{\G}_t-\nabla^2 f(\x_{t+1})$ to form $\U_t$.
}
\section{Conclusion}
\label{sec:conclusion}
In this paper, we have proposed rank-$k$ (\srk) methods for convex optimization.
We have proved that \srk methods enjoy the explicit local superlinear rate of $\OM\left((1-k/d)^{t(t-1)/2}\right)$.
Our result successfully reveals the advantage of block-type updates in quasi-Newton methods, building a bridge between the theories of ordinary quasi-Newton methods and the standard Newton method.
We also provide the convergence rate of $\OM\left((1-k/(\varkappa d))^{t(t-1)/2}\right)$ for randomized block BFGS/DFP methods and design a faster randomized block BFGS method to match the rate of \srk methods.
In future work, 
it would be interesting to study the global convergence of block quasi-Newton methods ~\citep{jiang2023online,jiang2023accquasi,rodomanov2024global,jin2024nona,jin2024nonb} and their limited memory~\citep{berahas2022limited,liu1989limited,gao2023limited} as well as stochastic variants~\citep{mokhtari2018iqn,wang2017stochastic,wang2019stochastic,yang2022stochastic}.

\section*{Acknowledgments}
We are grateful to Peter Richtárik for sharing with us his work on block quasi-Newton methods, which inspired some of the ideas developed in this paper.
Chengchang Liu is supported by the National Natural Science Foundation of China (No. 624B2125).
Luo Luo is supported by the National Natural Science Foundation of China (No. 12571557) and the Shanghai Basic Research Program (23JC1401000). 
Cheng Chen is supported by the National Natural Science Foundation of China (No. 62306116) and the Shanghai Special Fund for Promoting High-Quality Industrial Development (No. 20250307).

\appendix
\section{An Elementary Proof of Lemma~\ref{lm:explicitbound}(a)}
\label{appen:addiproof}
Before proving Lemma \ref{lm:explicitbound}(a), we first provide the following lemma.
\begin{lemma}\label{lem:1d_gauss}
Assume matrix $\mP\in\BR^{d\times r}$ is column orthonormal with $r\le d$ and the $d$-dimensional random vector $\vv$ is distributed according to $\mathcal{N}_d(\vzero,\mP\mP^{\top})$, then it holds that $\BE\left[{\vv\vv^\top}/{(\vv^\top \vv)}\right]=\mP\mP^\top/r$.
\end{lemma}
\begin{proof}
The distribution $\vv\sim\mathcal{N}_d(\vzero,\mP\mP^{\top})$ implies there exists a $r$-dimensional normally distributed random vector $\vw\sim\mathcal{N}_r(\vzero,\mI_r)$ such that $\vv=\mP\vw$. Thus, we have
\begin{align*}
\BE\left[\frac{\vv\vv^\top}{\vv^\top \vv}\right] 
= & \BE\left[\frac{(\mP\vw)(\mP\vw)^\top}{(\mP\vw)^\top (\mP\vw)}\right] 
=  \BE\left[\frac{\mP\vw\vw^\top\mP^\top}{\vw^\top\mP^\top\mP\vw}\right] 
=  \mP\BE\left[\frac{\vw\vw^\top}{\vw^\top\vw}\right]\mP^\top 
 = \frac{1}{r}\mP\mP^\top,
\end{align*}
where the last step is because $\vw/\norm{\vw}$ is uniformly distributed on the $r$-dimensional unit sphere and its covariance matrix is $\mI_r/r$.
\end{proof}

Now we prove statement (a) of Lemma~\ref{lm:explicitbound}.
\begin{proof}    
We only need to prove equation \eqref{eq:EP} by using induction on $k$.
The induction base $k=1$ have been verified by Lemma~\ref{lem:1d_gauss}. 
Now we assume equation~\eqref{eq:EP}
holds for any $\mU\in\BR^{d\times k}$ such that each of its entries are independently distributed according to $\fN(0,1)$. 
We define the random matrix 
\begin{align*}
\hat{\mU} = \begin{bmatrix}
\mU & \vv 
\end{bmatrix}\in\BR^{d\times(k+1)},
\end{align*}
where $\vv\sim\fN_d(\vzero,\mI_d)$ is independent distributed to $\mU$. We define $\B=\mU(\mU^\top\mU)^{-1}\mU^\top$, then we use block matrix inversion formula to rewrite $(\hat{\mU}^\top\hat{\mU})^{-1}$ as follows
\begin{align*}
  &(\hat{\mU}^\top\hat{\mU})^{-1} 
  =   \left(\begin{bmatrix}\mU^\top \\ \vv^\top \end{bmatrix} \begin{bmatrix}\mU & \vv \end{bmatrix} \right)^{-1} = \left(\begin{bmatrix}
      \U^{\top}\U&\U^{\top}\v\\
      \v^{\top}\U&\v^{\top}\v
  \end{bmatrix}\right)^{-1}\\[0.15cm]
  &=\begin{bmatrix}
  \S  & -(\U^{\top}\U)^{-1}\U^{\top}\v(\v^{\top}(\I_d-\B)\v)^{-1}\\
   -(\v^{\top}(\I_d-\B)\v)^{-1}\v^{\top}\U(\U^{\top}\U)^{-1}  &~ (\v^{\top}(\I_d-\B)\v)^{-1}
  \end{bmatrix},
\end{align*}
where $\S= (\U^{\top}\U)^{-1}+(\U^{\top}\U)^{-1}\U^{\top}\v(\v^{\top}(\I_d-\B)\v)^{-1}\v^{\top}\U(\U^{\top}\U)^{-1}$.
Thus, we have
\begin{align*}
     & \hat{\mU}(\hat{\mU}^\top\hat{\mU})^{-1}\hat{\mU}^\top
     = \begin{bmatrix}\mU & \vv \end{bmatrix}\left(\begin{bmatrix}\mU^\top \\ \vv^\top \end{bmatrix} \begin{bmatrix}\mU & \vv \end{bmatrix} \right)^{-1} \begin{bmatrix}\mU^\top \\ \vv^\top \end{bmatrix} \\
     &= \B+\frac{\B\v\v^{\top}\B}{\v^{\top}(\I_d-\B)\v}-\frac{\B\v\v^{\top}}{\v^{\top}(\I_d-\B)\v}-\frac{\v\v^{\top}\B}{\v^{\top}(\I_d-\B)\v}+\frac{\v\v^{\top}}{\v^{\top}(\I_d-\B)\v}\\
     &=\B+\frac{(\mI_d-\B)\vv\vv^\top(\mI_d-\B)}{\vv^\top(\mI_d-\B)\vv}.
\end{align*}
Since the rank of the projection matrix $\mI_d-\B$ is $d-k$, we can write  $\mI_d-\B=\mQ\mQ^\top$, where $\mQ\in\BR^{d\times (d-k)}$ is column orthonormal.
\begin{samepage}
Thus, we achieve
\begin{align*}
 & \EBP{\hat{\mU}(\hat{\mU}^\top\hat{\mU})
 \hat{\mU}^\top}
=\frac{k}{d}\mI_d +\BE_{\mU}\left[\BE_{\vv}\left[\frac{(\mI_d-\B)\vv\vv^\top(\mI_d-\B)}{\vv^\top(\mI_d-\B)\vv}\,\Big|\,\mU\right]\right]\\
&=\frac{k}{d}\mI_d +\BE_{\mU}\left[\BE_{\vv}\left[\frac{(\mQ\mQ^\top\vv)(\vv^\top\mQ\mQ^\top)}{(\vv^\top\mQ\mQ^\top)(\mQ\mQ^\top\vv)}\,\Big|\,\mU\right]\right]
=\frac{k}{d}\mI_d +\frac{1}{d-k}\BE_{\mU}[\mQ\mQ^\top] \\
&=\frac{k}{d}\mI_d +\frac{1}{d-k}\BE_{\mU}[\mI_d-\B] 
=\frac{k}{d}\mI_d +\frac{1}{d-k}\left(\mI_d-\frac{k}{d}\mI_d\right) 
=\frac{k+1}{d}\mI_d,
\end{align*}
\end{samepage}
which completes the induction.
In the above derivation, the first and the fifth equalities come from the inductive hypothesis; third equality is achieved by applying Lemma~\ref{lem:1d_gauss} with $\mP=\mQ\mQ^\top$ and $r=d-k$, and the fact~$\mQ\mQ^{\top}\vv\sim\mathcal{N}_d(\vzero,\mQ\mQ^{\top})$.
\end{proof}

\section{Auxiliary Lemmas for Non-negative Sequences}
The following lemmas on non-negative sequences
extend Theorem 23 of \cite{lin2021greedy} and Theorem 4.7 of \cite{rodomanov2021greedy}, leading to a general framework for the analysis of (block) quasi-Newton methods.

\begin{lemma}
\label{lm:superlinear}
Let $\{\lambda_t\}$ and $\{\delta_t\}$ be two non-negative random sequences that satisfy
\begin{align}
\label{eq:supercondi}
  &\BE_{t}\left[\delta_{t+1}\right]\leq \left(1-\frac{1}{\alpha}\right)(1+c_1\lambda_t)^2(\delta_t+c_2\lambda_t), ~~\lambda_{t+1}\leq (1+c_1\lambda_t)^2(\delta_t+c_3\lambda_t)\lambda_t,\\
\label{eq:supercondi_2}
 &\delta_0+c\lambda_0 \leq s,~~~\text{and}~~~ \lambda_{t}\leq \left(1-\frac{1}{\beta}\right)^t\lambda_0,
\end{align}
for some $c_1, c_2, c_3\geq 0$ with $c=\max\{c_2,c_3\}$, $s\geq 0$,  $\alpha>1$, and $\beta>1$, where we define $\EB_{t}[\,\cdot\,]\triangleq \EB[\,\cdot\,|\,\delta_0,\cdots,\delta_{t},\lambda_0,\cdots,\lambda_{t}]$. 
If $\lambda_0>0$ is sufficient small such that 
\begin{align}
\label{eq:supercondiini}
    \lambda_0\leq \frac{\ln 2}{\beta (2c_1+c(\alpha/(\alpha-1)))},
\end{align}
we have $\BE\left[{\lambda_{t+1}}/{\lambda_t}\right]\leq 2s\left(1-{1}/{\alpha}\right)^{t}.$
\end{lemma}

\begin{proof}
We denote $  \theta_t\triangleq \delta_t+c \lambda_t.$
Noticing that we have $\exp{x}\geq 1+x$ for all $x\geq 0$. Applying this fact with $x=c_1\lambda_t$ and the definition $c=\max\{c_2,c_3\}$, we have
\begin{align}
\label{eq:lambdat_leq_thetat}
  &\EB_{t}[\delta_{t+1}] \overset{\eqref{eq:supercondi}}{\leq}\left(1-\frac{1}{\alpha}\right)(\delta_t+c_2\lambda_t)(1+c_1\lambda_t)^2\leq \left(1-\frac{1}{\alpha}\right)\theta_t\exp{2c_1\lambda_t}\\
\label{eq:lambda_t_succ}
    &\text{and} \quad \lambda_{t+1} \overset{\eqref{eq:supercondi}}{\leq} (\delta_t+c_3\lambda_t)(1+c_1\lambda_t)^2\lambda_t \leq \theta_t\lambda_t\exp{2c_1\lambda_t}.
\end{align}
We denote $\tilde{c}\triangleq2c_1+c\alpha/(\alpha-1)$, then it holds that
\begin{align}
\label{eq:thetat+1leqs}
\begin{split}
    &\EB_{t}[\theta_{t+1}]= \BE_t[\delta_{t+1}+c \lambda_{t+1}] \\
    &\,\,\,\,\overset{\eqref{eq:lambdat_leq_thetat},\eqref{eq:lambda_t_succ}}{\leq} 
    \left(1-\frac{1}{\alpha}\right)\theta_t\exp{2c_1\lambda_t} + c\theta_t\lambda_t\exp{2c_1\lambda_t}\\
    &~~~~~=~~~\left(1-\frac{1}{\alpha}\right)\left(1+\frac{c\alpha \lambda_t}{\alpha-1}\right)\theta_t \exp{2c_1\lambda_t}
    \\
    &~~~~~\leq~~~\left(1-\frac{1}{\alpha}\right)\theta_t\exp{\left(2c_1 + \frac{c\alpha}{\alpha-1}\right)\lambda_t}\\
    &~~~~~=~~~\left(1-\frac{1}{\alpha}\right)\theta_t\exp{\tilde{c}\lambda_t}
    \overset{\eqref{eq:supercondi_2}}{\leq} \left(1-\frac{1}{\alpha}\right)\theta_t\exp{\tilde{c}\left(1-\frac{1}{\beta}\right)^t\lambda_0},
    \end{split}
\end{align}
where the second inequality comes from $\exp{x}\geq 1+x$ with $x=c\alpha\lambda_t/(\alpha-1)$.
Taking expectation on both sides of equation \eqref{eq:thetat+1leqs}, we have
\begin{align}
\label{eq:expthetat}
    \EB[\theta_{t+1}]\leq \left(1-\frac{1}{\alpha}\right)\exp{\tilde{c}\left(1-\frac{1}{\beta}\right)^t\lambda_0}\EB[\theta_t],
\end{align}
where we use the fact $\EB[\EB_{t}[\delta_{t+1}]] = \EB[\delta_{t+1}]$.
Therefore, we finish the proof as follows
\begin{align*}    
& \BE\left[\frac{\lambda_{t+1}}{\lambda_t}\right]
\overset{\eqref{eq:lambda_t_succ}}{\leq} \BE[\theta_{t}\exp{2c_1\lambda_t}] 
\overset{\eqref{eq:supercondi_2}}{\leq} \EBcommon{\theta_{{t}}}\exp{\tilde{c}\big(1-\frac{1}{\beta}\big)^t\lambda_0} 
\\
&\overset{\eqref{eq:expthetat}}{\leq}  \left(1-\frac{1}{\alpha}\right)\EBcommon{\theta_{t-1}}\exp{\tilde{c}\big(1-\frac{1}{\beta}\big)^{t-1}\lambda_0+\tilde{c}\big(1-\frac{1}{\beta}\big)^t\lambda_0}\\
&\overset{\eqref{eq:expthetat}}{\leq}\left(1-\frac{1}{\alpha}\right)^{t} \EBcommon{\theta_0}\exp{\tilde{c}\sum_{i=0}^{{t}}\big(1-\frac{1}{\beta}\big)^i\lambda_0}\\
&~~{\leq}~~  \left(1-\frac{1}{\alpha}\right)^t \EBcommon{\delta_0+c\lambda_0}\exp{\tilde{c}\beta\lambda_0}\overset{\eqref{eq:supercondi_2},\,\eqref{eq:supercondiini}}{\leq} \left(1-\frac{1}{\alpha}\right)^t 2s.
\end{align*}

\end{proof}

\begin{lemma}
\label{lm:linear}
Let~$\{\lambda_t\}$ and $\{\tilde{\eta}_t\}$ be two positive sequences with $\tilde{\eta}_t\geq 1$ and satisfy 
\begin{align}
\label{eq:lambda_t_1}
    &\lambda_{t+1}\leq \left(1-\frac{1}{\tilde{\eta}_t}\right)\lambda_t + \frac{m_1\lambda_t^2 }{2} + \frac{m_1^2\lambda_t^3}{4\tilde{\eta}_t}~~~\text{for all $\lambda_t$ such that $m_1\lambda_t\leq 2$}\\
    \label{eq:etat+1}
    &{\text {and}}~~~\tilde{\eta}_{t+1}\leq (1+m_2\lambda_t)^2\tilde{\eta}_t
\end{align}
for some $m_1,m_2>0$.
If
\begin{align}
    \label{eq:linear_initial}
    m\lambda_0\leq \frac{\ln ({3}/{2})}{4\tilde{\eta}_0},
\end{align}
where $m\triangleq\max\{m_1,m_2\}$,
then it holds that
\begin{align}
& \tilde{\eta}_t\leq\tilde{\eta}_0 \exp{2m\sum_{i=0}^{t-1}\lambda_i}\leq \frac{3\tilde{\eta}_0}{2} \label{eq:tilde_eta} \\
&\text{and}\quad
\lambda_t\leq\left(1-\frac{1}{2\tilde{\eta}_0}\right)^t\lambda_0. \label{eq:lambda_t_linear}
\end{align}
\end{lemma}
\begin{proof}
We prove the results of~\eqref{eq:tilde_eta} and \eqref{eq:lambda_t_linear} by induction. 
In the case of $t=0$, inequalities \eqref{eq:tilde_eta} and  \eqref{eq:lambda_t_linear} are satisfied naturally by conditions $\tilde\eta_t\geq 1$ and condition (\ref{eq:etat+1}), where we define $\sum_{i=0}^{t}\lambda_i=0$ if $t<0$. 
Now we suppose inequalities~\eqref{eq:tilde_eta} and \eqref{eq:lambda_t_linear} holds for $t=0,\dots,\hat{t}$, then we have
\begin{align}
\label{eq:bound-sum-lambdai}
m\sum_{i=0}^{\hat{t}}\lambda_i\overset{\eqref{eq:lambda_t_linear}}{\leq} m\lambda_0\sum_{i=0}^{\hat{t}}\left(1-\frac{1}{2\tilde{\eta}_0}\right)^{i} \leq  m \lambda_0 \cdot 2\tilde{\eta}_0 \overset{\eqref{eq:linear_initial}}{\leq} \frac{\ln(3/2)}{2}.
\end{align}
In the case of $t=\hat{t}+1$, we use induction to achieve
\begin{equation}
\label{eq:simple_induct}
\begin{split}
  \frac{ 1-{m_1\lambda_{\hat{t}}}/{2}}{\tilde{\eta}_{\hat{t}}}&\geq\frac{\exp{-m_1\lambda_{\hat{t}}}}{\tilde{\eta}_{\hat{t}}}\overset{\eqref{eq:tilde_eta}}{\geq} \frac{\exp{-m_1\lambda_{\hat{t}}}\cdot \exp{-2m\sum_{i=0}^{\hat{t}-1}\lambda_i}}{\tilde{\eta}_0}\\
  &\geq
  \frac{\exp{-2m\sum_{i=0}^{\hat{t}}\lambda_i}}{\tilde{\eta}_0}\overset{\eqref{eq:bound-sum-lambdai}}{\geq}  \frac{2}{3\tilde{\eta}_0},
\end{split}
\end{equation}  
where the first step is due to the fact $\exp{-2x}\leq 1-x$ with $x=m_1\lambda_{\hat t}/2\in[0,1/2]$ and the last second step is based on $m\geq m_1$. We also have 
\begin{equation} 
\label{eq:lambda_t_bound}
m_1\lambda_{\hat t} \leq m\lambda_{\hat{t}}\overset{\eqref{eq:lambda_t_linear}}{\leq} m\lambda_0\overset{\eqref{eq:linear_initial}}{\leq} \frac{1}{8\tilde{\eta}_0}\leq 2,
\end{equation}
where the last step is due to $\tilde\eta_0\geq 1$.
According to the condition \eqref{eq:lambda_t_1}, we have
\begin{align*}
&  \lambda_{\hat{t}+1}\overset{\eqref{eq:lambda_t_1}}{\leq} 
     \left(1-\frac{1}{\tilde{\eta}_{\hat{t}}}\right)\lambda_{\hat{t}} + \frac{m_1\lambda_{\hat{t}}^2 }{2} + \frac{m_1^2\lambda_{\hat{t}}^3}{4\tilde{\eta}_{\hat{t}}} =\left(1+\frac{m_1\lambda_{\hat{t}}}{2}\right)\cdot\left(1-\frac{1-{m_1\lambda_{\hat{t}}}/{2}}{\tilde{\eta}_{\hat{t}}}\right)\lambda_{\hat{t}}\\
&~\overset{\eqref{eq:simple_induct},\,\eqref{eq:lambda_t_bound}}{\leq}  \left(1+\frac{1}{16\tilde{\eta}_0}\right)\cdot\left(1-\frac{2}{3\tilde{\eta}_0}\right)\lambda_{\hat{t}} \leq \left(1-\frac{1}{2\tilde{\eta}_0}\right)\lambda_{\hat{t}}
\overset{\eqref{eq:lambda_t_linear}}{\leq} \left(1-\frac{1}{2\tilde{\eta}_0}\right)^{{\hat{t}}+1}\lambda_0.
\end{align*}
The induction also implies
\begin{align*}
   \tilde{ \eta}_{\hat{t}+1}\overset{\eqref{eq:etat+1}}{\leq}(1+m_2\lambda_{\hat{t}})^2\tilde{\eta}_{\hat{t}}\leq \tilde{\eta}_{\hat{t}} \exp{2m\lambda_{\hat{t}}}\overset{\eqref{eq:tilde_eta}}{\leq} \tilde{\eta}_0\exp{2m\sum_{i=0}^{\hat{t}}\lambda_{\hat{t}}}\overset{\eqref{eq:bound-sum-lambdai}}{\leq } \frac{3\tilde{\eta}_0}{2},
\end{align*}
where the second inequality is due to the fact $\exp{x}\geq 1+x$ with $x = m_2\lambda_{\hat{t}}\geq 0$. Thus, we finish the proof.
\end{proof}
\section{The Proof of Theorem~\ref{thm:srklinear}}
\label{sec:srklinear}
We first present a lemma from~\cite{lin2021greedy}.
\begin{lemma}[{\citet[Lemma 25]{lin2021greedy}}]
\label{lm:strong_self}
If the twice differentiable function $f(\cdot)$ satisfies Assumptions~\ref{ass:strongconvex} and \ref{ass:strongself}, and the positive-definite matrix~$\G\in\RB^{d\times d}$ and the vector $\vx\in\BR^d$ satisfy $\nabla^2 f(\x)\preceq \G\preceq \eta\nabla^2 f(\x)$ for some~$\eta>1$, then we have
\begin{align}
& \nabla^2 f(\x_{+})\preceq\tilde{\G}\preceq \eta(1+Mr)^2\nabla^2 f(\x_{+}) \label{eq:tilde_G_geq} \\
& \text{and}~~~\tau_{\nabla^2f(\x_{+})}(\tilde{\G})\leq (1+Mr)^2\left(\frac{\tau_{\nabla^2 f(\x)}(\G)}{\trcommon{\nabla^2f(\x)}}+2Mr\right)\tr{\nabla^2f(\x_{+})}    \label{eq:delta_SRK_lin}
\end{align}
for any $\vx_+\in\BR^d$, where $\tilde{\G}=(1+Mr)\G$ and $r=\|\x-\x_{+}\|_{\x}$.
\end{lemma}

Now, we prove Theorem~\ref{thm:srklinear} by using Lemma~\ref{lm:sr1good}, \ref{lm:linear-quadra}, \ref{lm:linear}, and \ref{lm:strong_self}.

\begin{proof}
Denote $\tilde{\eta}_t \triangleq \min_{\G_t\preceq \eta\nabla^2 f(\x_t)}\eta$. We first prove that for all $t\geq 0$, it holds
\begin{align}\label{psd:G_t}
    \nabla^2 f(\x_t)\preceq \G_t\preceq \tilde{\eta}_t\nabla^2 f(\x_t).
\end{align}
The result of $\G_t\preceq \tilde{\eta}_t\nabla^2 f(\x_t)$ holds by the definition of $\tilde{\eta}_t$, and we prove~$\nabla^2 f(\x_t)\preceq \G_t$ by induction.
For the case of $t=0$, it holds naturally by the initial condition. Suppose it holds for $t=0,1,\dots,\hat{t}$. Then for the case of $t=\hat{t}+1$, we have
\begin{align*}
\G_{\hat{t}+1}=\srk(\tilde{\G}_{\hat{t}},\nabla^2 f(\x_{\hat{t}+1}),\U_{\hat{t}})\overset{\eqref{eq:srk_good}}{\succeq} \tilde{\G}_{\hat{t}} \overset{\eqref{eq:tilde_G_geq}}{\succeq} \nabla^2  f(\x_{\hat{t}+1}),
\end{align*}
which finishes the induction.

Let $\lambda_t\triangleq\lambda(\x_t)$. 
Applying equation (\ref{psd:G_t}) and Lemma~\ref{lm:linear-quadra}, we achieve
\begin{align*}
        \lambda_{t+1}\leq \left(1-\frac{1}{\tilde{\eta}_t}\right)\lambda_t + \frac{M\lambda_t^2}{2} + \frac{M^2\lambda_t^3}{4\tilde{\eta}_t}~~~\text{for all $\lambda_t$ such that $M\lambda_t\leq 2$}.
\end{align*}
We then figure out the relation between $\tilde{\eta}_{t+1}$ and $\tilde{\eta}_t$. 
Let $r_t=\|\x_{t+1}-\x_t\|_{\x_t}$.
Applying equation (\ref{psd:G_t}) and Lemma~\ref{lm:strong_self} with $\G=\G_t$, $\x=\x_t$, $\x_{+}=\x_{t+1}$, $\eta = \tilde{\eta}_t$, and $r=r_t$, we achieve
\begin{align}
\label{eq:tildeG_t}
      \nabla^2 f(\x_{t+1})\preceq\tilde{\G}_t\preceq \tilde{\eta}_t(1+Mr_t)^2\nabla^2 f(\x_{t+1}).
\end{align}
Using Lemma~\ref{lm:sr1good} with $\G=\tilde{\G}_t$, $\A= \nabla^2 f(\x_{t+1})$, and $\eta=\tilde{\eta}_t(1+Mr_t)^2$, we have
\begin{align*}   
& \G_{t+1}
\overset{\eqref{eq:srk_good},\,\eqref{eq:tildeG_t}}{\preceq} \tilde{\eta}_t(1+Mr_t)^2\nabla^2 f(\x_{t+1})\overset{\eqref{eq:linear-quadra}}{\preceq} \tilde{\eta}_t(1+M\lambda_t)^2\nabla^2 f(\x_{t+1}),
\end{align*}
which means
$\tilde{\eta}_{t+1}=\min_{\G_{t+1}\preceq \eta \nabla^2 f(\x_{t+1})}\eta\leq \tilde{\eta}_t(1+M\lambda_t)^2$.
Hence, sequences $\{\tilde{\eta}_t\}$ and~$\{\lambda_t\}$ satisfy the conditions of Lemma~\ref{lm:linear} with $m_1=m_2=M$. We then apply Lemma~\ref{lm:linear} to obtain
\begin{align*}
 \nabla^2 f(\x_t)\preceq \tilde{\G}_t \preceq \frac{3\tilde{\eta}_0}{2}\nabla^2f(\x_t)\preceq\frac{3{\eta}_0}{2}\nabla^2 f(\x_t)
\end{align*}
and
\begin{align*}
 \lambda(\x_t)\leq \left(1-\frac{1}{2\tilde{\eta}_0}\right)^t\lambda(\x_0)\leq\left(1-\frac{1}{2{\eta}_0}\right)^t\lambda(\x_0).
\end{align*}
\end{proof}

\section{The Proof of Theorem~\ref{thm:srk}}
\label{sec:srk_proof}
We first provide some auxiliary lemmas that will be used in our later proof.

\begin{lemma}
\label{lm:GHneq}
For any positive definite symmetric matrices $\G,\H\in\RB^{d\times d}$ such that~$\H\preceq\G$, 
it holds that
\begin{align}
\label{eq:GHneq}
\G\preceq \left(1+\frac{\hat\varkappa d \tau_{\H}(\G)}{\trcommon{\H}}\right)\H,
\end{align}
where $\hat\varkappa$ is the condition number of $\H$ and the notation of $\tau_\H(\G)$ follows the expression of \eqref{eq:measure_srk}.
If it further satisfies that $\G\preceq\eta\H$ for some $\eta\geq 1$, then we have
\begin{align}
\label{eq:GHneq_2}
  \qquad\frac{\tau_{\H}(\G)}{{\rm tr}(\H)} \leq \eta -1.
\end{align}
\end{lemma}
\begin{proof}
We obtain inequality \eqref{eq:GHneq} by statement 3) of equation (42) from Lin et al. \cite[Lemma 25]{lin2021greedy}.
We prove inequality \eqref{eq:GHneq_2} by the definition of $\tau_{\H}(\G)$ as follows
\begin{align*}
     \frac{\tau_{\H}(\G)}{\trcommon{\H}} =\frac{ \trcommon{\G-\H}}{\trcommon{\H}}{\leq} \frac{\trcommon{(\eta-1)\H}} {\trcommon{\H}}= \eta-1.
\end{align*}

\end{proof}

\begin{lemma}[{\citet[Lemma 26]{lin2021greedy}}]
\label{lm:randomsequence_lin}
Suppose the non-negative random sequence $\{X_t\}$
satisfies $\EB[X_t]\leq c\left(1-{1}/{\alpha}\right)^t$ for all $t\geq0$ and some constants $c\geq 0$ and $\alpha>1$. 
Then for any $\delta\in(0,1)$, we have $  X_t\leq {c\alpha^2}\left(1-{1}/{(1+\alpha)}\right)^t/\delta$
for all $t$ with probability at least~$1-\delta$.
\end{lemma}

Now, we provide the proof of Theorem~\ref{thm:srk}.

\begin{proof}
We denote 
$\EB_{t}[\,\cdot\,]\triangleq \EB[\,\cdot\,|\,\U_0,\cdots,\U_{t-1}]$, 
$g_t={\tau_{\nabla^2 f(\x_t)}(\G_t)}/{\trcommon{\nabla^2 f(\x_t)}}$, 
$\lambda_t=\lambda(\x_t)$, $r_t=\|\x_{t+1}-\x_t\|_{\x_t}$,
and $\delta_t=d\varkappa g_t$.
The initial condition \eqref{eq:initial} means we can apply
Theorem~\ref{thm:srklinear} to achieve
\begin{align}\label{eq:linearrate}
    \nabla^2 f(\x_t)\preceq\G_t
    \qquad\text{and}\qquad
    \lambda_{t}\leq \left(1-\frac{1}{2\eta_0}\right)^{t}\lambda_0.
\end{align}
According to Theorem~\ref{thm:matrix}, we have
\begin{align}
\label{eq:srk_G-H}
    \EB_{t}\left[\tau_{\nabla^2 f(\x_{t+1})}(\G_{t+1})\right] \overset{\eqref{eq:measure_srk}}{\leq} \left(1-\frac{k}{d}\right) \tau_{\nabla^2 f(\x_{t+1})}(\tilde{\G}_{t}).
\end{align}
According to Lemma~\ref{lm:strong_self} with $\G=\G_t$,  $\x=\x_t$, $\x_{+}=\x_{t+1}$, and $r=r_t$, we have
\begin{align}
\label{eq:tau_neq}
    \tau_{\nabla^2f(\x_{t+1})}(\tilde{\G}_t)\overset{\eqref{eq:delta_SRK_lin}}{\leq} (1+Mr_t)^2(g_t + 2Mr_t)\trcommon{\nabla^2f(\x_{t+1})},
\end{align}
which means
\begin{align}
\label{eq:delta_t_condi}
\begin{split}
      & \EB_{t}[\delta_{t+1}] = \EB_{t}\left[\frac{d\varkappa \tau_{\nabla^2 f(\x_{t+1})}(\G_{t+1})}{\trcommon{\nabla^2 f(\x_{t+1})}}\right] 
       \overset{\eqref{eq:srk_G-H}}{\leq} \left(1-\frac{k}{d}\right)\left(\frac{d\varkappa \tau_{\nabla^2 f(\x_{t+1})}(\tilde{\G}_t)}{\trcommon{\nabla^2 f(\x_{t+1})}}\right)\\
      &\overset{\eqref{eq:tau_neq}}{\leq} \left(1-\frac{k}{d}\right)\left(\frac{d\varkappa(1+Mr_t)^2(g_t + 2Mr_t)\trcommon{\nabla^2f(\x_{t+1})} }{\trcommon{\nabla^2 f(\x_{t+1})}}\right)\\
      &\overset{\eqref{eq:linear-quadra}}{\leq}\left(1-\frac{k}{d}\right)(1+M\lambda_t)^2(\delta_t+2\varkappa dM\lambda_t).
\end{split}
\end{align}
According to Lemma~\ref{lm:GHneq} with $\G=\G_t$ and $\H=\nabla^2 f(\x_t)$, we have
\begin{align*}
    \nabla^2 f(\x_{t})\overset{\eqref{eq:linearrate}}{\preceq} \G_t \overset{\eqref{eq:GHneq}}{\preceq} (1+\delta_t)\nabla^2 f(\x_t).
\end{align*}
According to Lemma~\ref{lm:linear-quadra} with $\eta_t=1+\delta_t$, we have
 \begin{align}
 \label{eq:lambda_t_condi}
\begin{split}
        & \lambda_{t+1}
        \overset{\eqref{eq:linear-quadra}}{\leq} \left(1-\frac{1}{1+\delta_t}\right)\lambda_t +\frac{M\lambda_t^2}{2} + \frac{M^2\lambda_t^3}{4(1+\delta_t)}\\ 
        &= \left(1+\frac{M\lambda_t}{2}\right)\cdot\left(\frac{\delta_t + {M\lambda_t}/{2}}{1+\delta_t}\right)\lambda_t  \,\,\leq\left(1+M\lambda_t\right)^2\left(\delta_t+\frac{M\lambda_t}{2} \right)\lambda_t.
\end{split}
 \end{align}
According to Lemma \ref{lm:GHneq} with $\G=\G_0$, $\H=\nabla^2 f(\x_0)$, $\eta=\eta_0$, and the initial condition \eqref{eq:initial}, we have
\begin{align}
\label{eq:delta_0_bound}
    &\delta_0=\frac{d\varkappa\tau_{\nabla^2 f(\x_0)}(\G_0)}{\trcommon{\nabla^2 f(\x_0)}}\overset{\eqref{eq:GHneq_2}}{\leq}d\varkappa(\eta_0-1)\\
    \label{eq:initial_condi_theta}
    & \text{and}~~~ \delta_0+2d\varkappa M\lambda_0\overset{\eqref{eq:delta_0_bound}}{\leq} d\varkappa (\eta_0 -1) +2d\varkappa M\lambda_0\overset{\eqref{eq:initial}}{\leq}  d\varkappa\eta_0. 
\end{align}
Then we apply Lemma~\ref{lm:superlinear} 
on the random sequences $\{\lambda_t\}$ and $\{\delta_t\}$ with
$c_1=M$, $c_2={M}/{2}$, $c_3=2\varkappa dM$, $\alpha={d}/{k}$, $\beta=2\eta_0$, and $s=d\varkappa\eta_0$
to achieve inequality~\eqref{eq:E_lambda_srk}, where we can verify conditions \eqref{eq:supercondi} and \eqref{eq:supercondi_2} by equations \eqref{eq:delta_t_condi}, \eqref{eq:lambda_t_condi} and \eqref{eq:linearrate}, \eqref{eq:initial_condi_theta} respectively.

Now, we prove the two-stage convergence of \srk~methods as follows. 
\begin{enumerate}[label=(\alph*),topsep=0.05cm, itemsep=0.1cm, leftmargin=0.9cm] 
\item For \srk~method with randomized strategy such that $\left[\U_{t}\right]_{ij}\overset{\rm{i.i.d}}{\sim}{\fN(0,1)}$, we apply Lemma~\ref{lm:randomsequence_lin} with $X_t=\lambda_{t+1}/\lambda_t$, $\alpha = d/k$ and $c=2d\varkappa\eta_0$ to obtain 
 \begin{align}
 \label{eq:rasrk_super}
    \frac{ \lambda_{t+1}}{\lambda_t}  \leq \frac{2d^3\varkappa\eta_0}{k^2\delta}\left(1-\frac{k}{d+k}\right)^{t}
 \end{align}
holds for all $t$ with probability at least $1-\delta$. 
We take $t_0= \OM(d\ln(d\varkappa \eta_0/\delta)/k)$, which leads to
\begin{align}
\label{eq:rasrk_t0}
    \frac{2d^3\varkappa\eta_0}{k^2\delta}\left(1-\frac{k}{d+k}\right)^{t_0}\leq\frac{1}{2}.
\end{align}
Together with the linear convergence \eqref{eq:srklinear} provided by Theorem \ref{thm:srklinear}, we have
\begin{equation*}
\begin{split}
    &\lambda_{t+t_0}
    \overset{\eqref{eq:rasrk_super}}{\leq} \frac{2d^3\varkappa\eta_0}{k^2\delta}\cdot\left(1-\frac{k}{d+k}\right)^{t+t_0-1}\cdot\lambda_{t+t_0-1} \\
    & \!\!\overset{\eqref{eq:rasrk_t0}}{\leq} \left(1-\frac{k}{d+k}\right)^{t-1}\!\!\cdot\frac{\lambda_{t+t_0-1}}{2}\leq \cdots \overset{\eqref{eq:rasrk_super},\,\eqref{eq:rasrk_t0}}{\leq}\!\!\left(1-\frac{k}{d+k}\right)^{t(t-1)/2}\!\!\cdot\left(\frac{1}{2}\right)^{t}\lambda_{t_0}\\
    &\overset{\eqref{eq:srklinear}}{\leq} \left(1-\frac{k}{d+k}\right)^{t(t-1)/2}\cdot\left(\frac{1}{2}\right)^t\cdot\left(1-\frac{1}{2\eta_0}\right)^{t_0}\lambda_0
\end{split}
\end{equation*}
with probability at least $1-\delta$.

\item 
For \srk~method with greedy strategy such that $\U_t=\E_k(\tilde{\G}_t-\nabla^2 f(\x_{t+1}))$, we can take  $t_0=\OM\left(d\ln(\eta_0d\varkappa )/k\right)$ such that
\begin{align}
\label{eq:grsrk_t0}
  2d\varkappa\eta_0  \left(1-\frac{k}{d}\right)^{t_0}\leq \frac{1}{2}.
\end{align}
Together with the linear convergence \eqref{eq:srklinear} provided by Theorem \ref{thm:srklinear}, we have
\begin{align*}
    & \lambda_{t+t_0} \overset{\eqref{eq:E_lambda_srk}}{\leq} 2d\varkappa\eta_0\left(1-\frac{k}{d}\right)^{t+t_0-1}\lambda_{t+t_0-1} \\
    & \overset{\eqref{eq:grsrk_t0}}{\leq} \left(1-\frac{k}{d}\right)^{t-1}\cdot\frac{\lambda_{t+t_0-1}}{2}\leq \cdots \overset{\eqref{eq:E_lambda_srk},\,\eqref{eq:grsrk_t0}}{\leq} \left(1-\frac{k}{d}\right)^{t(t-1)/2}\cdot\left(\frac{1}{2}\right)^{t}\lambda_{t_0}\\
    &\,\overset{\eqref{eq:srklinear}}{\leq} \left(1-\frac{k}{d}\right)^{t(t-1)/2}\cdot\left(\frac{1}{2}\right)^t\cdot\left(1-\frac{1}{2\eta_0}\right)^{t_0}\lambda_0.
\end{align*}
\end{enumerate}
 \end{proof}

\section{The Proof of Theorem~\ref{thm:BFGS}}
\label{sec:bfgs_proof}
We first provide the linear convergence for the proposed block BFGS and DFP methods (Algorithm~\ref{alg:bfgs}).
\begin{lemma}
\label{lm:BFGSlinear}
Under the setting of Theorem~\ref{thm:BFGS},  Algorithm~\ref{alg:bfgs} holds that 
\begin{align}
\label{eq:BFGSlinear}
    \lambda(\vx_t)\leq \left(1-\frac{1}{2\eta_0}\right)^t\lambda(\vx_0)~~~\text{and}~~~ \nabla^2 f(\x_t)\preceq \G_t\preceq \frac{3\eta_0}{2}\nabla^2 f(\x_t)
\end{align}
for all $t\geq 0$.
\end{lemma}

\begin{proof}
We can obtain this result by following the proof of Theorem~\ref{thm:srklinear} by replacing the steps of using Lemma~\ref{lm:sr1good} with using Lemma~\ref{lm:bfgsnofar}.
\end{proof}

We then provide some auxiliary lemmas for later analysis.
\begin{lemma}[\!\!{\cite[Lemma 4.8]{rodomanov2021greedy}}]
\label{lm:strong_self_bfgs}
If the twice differentiable function $f:\BR^d\to\BR$ is $M$-strongly self-concordant and $\mu$-strongly convex and the positive definite matrix $\G\in\RB^{d\times d}$ and $\vx\in\BR^d$ satisfy $\nabla^2 f(\x)\preceq \G\preceq \eta\nabla^2 f(\x)$ for some $\eta>1$, then we have
\begin{align}
 \sigma_{\nabla^2f(\x_{+})}(\tilde{\G})\leq (1+Mr)^2 (\sigma_{\nabla^2f(\x)}({\G})+2dMr)   \label{eq:delta_BFGS_lin}
\end{align}
for any $\vx,\vx_{+}\in\BR^d$ where $\tilde{\G}=(1+Mr)\G$ and $r=\|\x-\x_{+}\|_{\x}$. 
\end{lemma}

\begin{lemma}\label{lm:GHneq_bfgs}
For any positive definite symmetric matrices $\G,\H\in\RB^{d\times d}$ such that~$\H\preceq\G$, 
we have
\begin{align}
\label{eq:GHneq_bfgs}
\G\preceq (1+\sigma_{\H}(\G))\H.
\end{align}
If it further satisfies that $\G\preceq\eta\H$ for some $\eta \geq 1$, then we have
\begin{align}
\label{eq:GHneq_2_bfgs}
  \sigma_{\H}(\G)\leq d(\eta -1).
\end{align}
\end{lemma}
\begin{proof}
We obtain inequality \eqref{eq:GHneq_bfgs} by statement 1) of Equation (42) by Lin et al. \cite[Lemma 25]{lin2021greedy}.
We prove inequality \eqref{eq:GHneq_2_bfgs} by the definition of $\sigma_{\H}(\G)$ as follows
\begin{align*}
      \sigma_{\H}(\G)&=\tr{\H^{-1}(\G-\H)}=\tr{\H^{-1/2}(\G-\H)\H^{-1/2}}\\
      &\leq \tr{(\eta-1)\H^{-1/2}\H\H^{-1/2}} = d(\eta-1).
\end{align*}
\end{proof}

Now, we present the proof for Theorem~\ref{thm:BFGS}.
\begin{proof}
Denote $\delta_t\triangleq\sigma_{\nabla^2 f(\x_t)}(\G_t)$ and $\lambda_t\triangleq \lambda(\x_t)$. 
According to Lemma~\ref{lm:BFGSlinear} and Lemma~\ref{lm:GHneq_bfgs} with $\mG=\mG_t$ and $\mH=\nabla^2 f(\vx_t)$, we have
\begin{align}
\label{eq:BFGScondi_GH}
   \nabla^2 f(\x_t)\overset{\eqref{eq:BFGSlinear}}{\preceq} \G_t \overset{\eqref{eq:GHneq_bfgs}}{\preceq} (1+\delta_t) \nabla^2 f(\x_t).
\end{align}
According to Theorem~\ref{thm:bfgs} with $\mA=\nabla^2 f(\vx_{t+1})$ and $\mG=\tilde\mG_t$, we have
\begin{align}
\label{eq:BFGSconverge1}
    \EB_t[\delta_{t+1}] = \EB_t\big[\sigma_{\nabla^2 f(\x_{t+1})}(\G_{t+1})\big] \overset{\eqref{eq:bfgssigma}}{\leq} \left(1-\frac{k}{d\varkappa}\right)\sigma_{\nabla^2 f(\x_{t+1})}(\tilde{\G}_t).
\end{align}
According to Lemma~\ref{lm:strong_self_bfgs} with $\mG=\mG_t$ and $\vx=\vx_t$, we have
\begin{align}
\label{eq:BFGSconverge2}
  \sigma_{\nabla^2 f(\x_{t+1})}(\tilde{\G}_t) \overset{\eqref{eq:delta_BFGS_lin}}{\leq}(1+Mr_t)^2(\delta_t+2dMr_t).
\end{align}
Thus, we can obtain following result
\begin{align}
\label{eq:delta_t_condi_BFGS}
\begin{split}
    &\EB_{t}[\delta_{t+1}] \overset{\eqref{eq:BFGSconverge1}}{\leq} \left(1-\frac{k}{d\varkappa}\right)\sigma_{\nabla^2 f(\x_{t+1})}(\tilde{\G}_{t})\\
    &\overset{\eqref{eq:BFGSconverge2}}{\leq}
    \left(1-\frac{k}{d\varkappa}\right) (1+Mr_t)^2(\delta_t+2dMr_t)\\
    &\overset{\eqref{eq:linear-quadra}}{\leq} \left(1-\frac{k}{d\varkappa}\right)(1+M\lambda_t)^2(\delta_t+2dM\lambda_t).
\end{split}
\end{align}
We follow Lemma~\ref{lm:linear-quadra} with $\eta_t = 1+\delta_t$ and the derivation of equation \eqref{eq:lambda_t_condi} to achieve
\begin{align}\label{eq:lambda_t_condi_BFGS}
    \lambda_{t+1} \leq (1+M\lambda_t)^2\left(\delta_t+\frac{M\lambda_t}{2}\right)\lambda_t.
\end{align}
According to Lemma~\ref{lm:BFGSlinear}, we have 
\begin{align}
\label{eq:lambda_linear_condi_BFGS}
    \lambda_{t}\leq \left(1-\frac{1}{2\eta_0}\right)^{t}\lambda_0.
\end{align}
According to Lemma~\ref{lm:GHneq_bfgs} with $\G=\G_0$,  $\H=\nabla^2 f(\x_0)$, $\eta=\eta_0$, and the initial condition (\ref{eq:bfgsini}), we have 
\begin{align}
\label{eq:theta_0_condi_BFGS}
    \delta_0=\sigma_{\nabla^2 f(\x_0)}(\G_0) \overset{\eqref{eq:GHneq_2_bfgs}}{\leq} d(\eta_0-1)\qquad\text{and}\qquad\delta_0+2dM\lambda_0\overset{\eqref{eq:bfgsini}}{\leq} d\eta_0.
\end{align}
Then we apply Lemma~\ref{lm:superlinear} 
on the random sequences $\{\lambda_t\}$ and $\{\delta_t\}$ with
$c_1=M$, $c_2=M/2$, $c_3=2dM$, $\alpha=d\varkappa/k$, $\beta={2\eta_0}$, and $s=d\eta_0$ to finish the proof, where we can verify conditions \eqref{eq:supercondi} and \eqref{eq:supercondi_2} by equations \eqref{eq:delta_t_condi_BFGS}, \eqref{eq:lambda_t_condi_BFGS} and \eqref{eq:lambda_linear_condi_BFGS}, \eqref{eq:theta_0_condi_BFGS} respectively.
\end{proof}

\section{The proof of Theorem \ref{thm:fasterBFGS}}
\label{app:fasterBFGSproof}
\begin{proof}
We define $\delta_t\triangleq\sigma_{\nabla^2 f(\x_{t})}(\G_{t})$ and $\lambda_t\triangleq \lambda(\x_t)$.   
The proof of this theorem can follow the counterpart of Theorem~\ref{thm:BFGS} by using  $\tilde{\LL}_{t}^{\top}\U_t$ to replace $\mU_t$ (line \ref{line:fastBFGS} in Algorithm~\ref{alg:fasterbfgs}). 
This leads to the faster rate for matrix approximation, i.e., applying Theorem \ref{thm:bfgs-1} with $\G=\tilde{\G}_t$, $\A = \nabla^2 f(\x_{t+1})$, and $\mU=\mU_t$ to achieve
\begin{align}\label{eq:fasterBFGSconverge1}
    \EB_t[\delta_{t+1}] = \EB_t[\sigma_{\nabla^2 f(\x_{t+1})}(\G_{t+1})] \overset{\eqref{eq:fasterbfgssigma-1}}{\leq} \left(1-\frac{k}{d}\right)\sigma_{\nabla^2 f(\x_{t+1})}(\tilde{\G}_t).
\end{align}
We follow the proof of Theorem~\ref{thm:BFGS}
by replacing equation~\eqref{eq:BFGSconverge1} with  \eqref{eq:fasterBFGSconverge1} and replacing all the terms of $\varkappa d$ with $d$ to finish the proof of Theorem \ref{thm:fasterBFGS}.
\end{proof}

\bibliography{references}

\end{document}